\documentclass[11pt,letterpaper]{article}
\usepackage{amsfonts, amsmath, amssymb, amscd, amsthm, color, graphicx, mathrsfs, wasysym, setspace, mdwlist, calc, float, mathtools, tikz-cd}

\usepackage{tocloft}
\usepackage{xcolor}

\usepackage{hyperref}
\hypersetup{linktocpage}

\hypersetup{colorlinks,
    linkcolor={red!50!black},
    citecolor={blue!80!black},
    urlcolor={blue!80!black}}
\usepackage{float}

\renewcommand{\phi}{\varphi}

\newcommand{\Aut}{Aut_f(F_\infty)}
\newcommand{\NN}{\mathbb{N}}
\newcommand{\ZZ}{\mathbb{Z}}
\renewcommand{\L}{\mathcal{L}}
\renewcommand{\S}{\mathcal{S}}

\newcommand{\G}{\mathcal G}

\newcommand{\ModS}{{\rm Mod}_{\mathcal S}}
\newcommand{\ModG}{{\rm Mod}_{\mathcal G}}
\newcommand{\Con}{{\rm Con}}
\newcommand{\e}{\varepsilon}

\renewcommand{\d}{{\rm d}}
\renewcommand{\ll }{\langle\hspace{-.7mm}\langle }
\newcommand{\rr }{\rangle\hspace{-.7mm}\rangle }
\newcommand{\Gxh}{\Gamma (G,\mathcal A)}
\newcommand{\h}{\hookrightarrow_h}
\newcommand{\Hk}{\{ H_1, \ldots, H_k\}}

\newtheorem{thm}{Theorem}[section]
\newtheorem*{thm*}{Theorem}
\newtheorem{cor}[thm]{Corollary}
\newtheorem{lem}[thm]{Lemma}
\newtheorem{prop}[thm]{Proposition}
\newtheorem{prob}[thm]{Question}

\theoremstyle{definition}
\newtheorem{defn}[thm]{Definition}
\newtheorem{conv}[thm]{Convention}
\newtheorem{ex}[thm]{Example}
\theoremstyle{remark}
\newtheorem{rem}[thm]{Remark}

\newcommand\blfootnote[1]{%
  \begingroup
  \renewcommand\thefootnote{}\footnote{#1}%
  \addtocounter{footnote}{-1}%
  \endgroup
}

\newfont{\eufm}{eufm10}

\begin{document}
\title{\vspace*{-17mm} A topological zero-one law and elementary equivalence of
finitely generated groups}

\author{D. Osin\thanks{This work has been supported by the NSF grant DMS-1612473.}}
\date{}

\maketitle
\vspace*{-6mm}

\begin{abstract} 
Let $\G$ denote the space of finitely generated marked groups. We give equivalent characterizations of closed subspaces $\S\subseteq \G$ satisfying the following zero-one law: for any sentence $\sigma$ in the infinitary logic $\mathcal L_{\omega_1, \omega}$, the set of all models of $\sigma$ in $\S$ is either meager or comeager. In particular, we prove that the zero-one law holds for certain natural spaces associated to hyperbolic groups and their generalizations. As an application, we show that generic torsion-free lacunary hyperbolic groups are elementarily equivalent; the same claim holds for lacunary hyperbolic groups without non-trivial finite normal subgroups. Our paper has a substantial expository component. We give streamlined proofs of some known results and survey ideas from topology, logic, and geometric group theory relevant to our work. We also discuss some open problems.
\end{abstract}

\tableofcontents

\section{Introduction}

The well-known zero-one law in model theory asserts that a first-order sentence is either almost surely true or almost surely false for a randomly chosen finite relational structure. More precisely, let $\L$ be a first-order relational language with finitely many predicates. Given $n\in \mathbb N$ and an $\L$-sentence $\sigma$, let $\mu_n(\sigma)$ denote the fraction of $\L$-structures of cardinality $n$ satisfying $\sigma$.

\begin{thm} \label{01FM}
For any $\L$-sentence $\sigma$, we have $\lim\limits_{n\to \infty}\mu_n(\sigma) \in \{ 0,1\}$.
\end{thm}

This theorem was first proved by Glebskii, Kogan, Liogonkii, and Talanov \cite{GKLT} in 1969 and rediscovered by Fagin \cite{Fag} in 1975. Since then, a plethora of generalizations and alternative versions of this result for infinitary logics and different probability models have been obtained (see \cite{GGM,KV92,KV00} and references therein).

It is natural to ask whether an analogue of Theorem \ref{01FM} holds for some infinite structures. In these settings, the property of having full measure can be replaced with genericity (in the Baire category sense) in an appropriate topological space. Our main goal is to study this question for finitely generated groups. We give several equivalent characterizations of classes of finitely generated groups for which a topological analogue of Theorem \ref{01FM} holds and discuss some applications, mostly to groups with hyperbolic properties.

This paper brings together ideas from topology, logic, and geometric group theory. To promote the interaction between these areas and to make our paper accessible for a wider audience, we include detailed proofs of some known results and give a brief survey of relevant background material. Many of our results can also be proved for finitely generated algebras of any countable signature. However, dealing with general notation makes the proofs more mysterious than they really are and we do not see any interesting applications outside of group theory at the moment. For this reason, we focus on groups throughout the main body of the paper and discuss possible generalizations in the last section.

\paragraph{Acknowledgments.} The author is grateful to Simon Thomas for fruitful discussions of this work. The author would also like to thank Arnaud Golfouse, M\'arton Elekes, and Yves Stalder, who pointed out mistakes and inaccuracies in the earlier versions of the paper.

\section{Main results}

We begin by briefly recalling the definition of the space of finitely generated marked groups introduced by Grigorchuk in \cite{Gri}.
Let $\G_n$ denote the set of all pairs $(G,A)$, where $G$ is a group and $A=(a_1, \ldots, a_n)$ is an ordered generating set of $G$, considered up to the following equivalence relation: two pairs $(G, (a_1, \ldots, a_n))$ and $(H, (b_1, \ldots , b_n))$ are identified if the map $a_1\mapsto b_1,\; \ldots,\; a_n\mapsto b_n$ extends to an isomorphism $G\to H$. For simplicity, we keep the notation $(G,A)$ for the equivalence class of $(G,A)$. A sequence $\{ (G_i, A_i)\}_{i\in \NN}$ converges to $(G,A)$ in $\G_n$ if the Cayley graphs $\Gamma (G_i,A_i)$ locally converge to $\Gamma (G,A)$ in the natural sense (see Section 3.3 for more details).

Identifying $(G, (a_1, \ldots , a_n))$ with $(G, (a_1, \ldots , a_n,1))$, we obtain an embedding $\G_n\subseteq \G_{n+1}$. The topological union $\G=\bigcup_{n\in \NN}\G_n$ is a totally disconnected Polish space called the \emph{space of finitely generated marked groups}.

For a finitely generated group $G$, we denote by $[G]\subseteq \G $ its \emph{isomorphism class}; that is,
$$
[G]=\{ (H,B)\in \G\mid H\cong G \}.
$$
A subset $\S\subseteq \G$ is said to be \emph{isomorphism-invariant} if $[G]\subseteq \S$ whenever $(G,A)\in \S$ for some finite generating set $A$.

Henceforth, let $\L$ denote the \emph{language of groups}. That is, $\L$ is the first-order language with the signature $\{ 1, \cdot, ^{-1}\}$. Recall that $\L_{\omega_1, \omega}$ is the infinitary version of $\L$,  where countable conjunctions and disjunctions (but only finite sequences of quantifiers) are allowed. It is useful to keep in mind that every countable theory in $\L_{\omega_1, \omega}$ is equivalent to a single sentence; in particular, every theory in the first-order logic is equivalent to an $\L_{\omega_1, \omega}$-sentence.

The expressive power of $\L_{\omega_1, \omega}$ is much greater than that of the first-order logic. In fact, most algebraic, geometric, and even analytic properties of groups (e.g., finiteness, solvability, hyperbolicity, amenability, property (T) of Kazhdan,  the properties of having exponential, subsexponential, and polynomial growth, etc.) can be expressed by $\L_{\omega_1, \omega}$-sentences.

\begin{defn}
We say that a subspace $\S\subseteq \G$  satisfies the \emph{zero-one law for $\mathcal L_{\omega_1, \omega}$-sentences} if for any $\mathcal L_{\omega_1, \omega}$-sentence $\sigma$, the \emph{set of models} $${\rm Mod}_{\S}(\sigma)= \{ (G,A)\in \S\mid G\models \sigma\}$$
is either meager or comeager in $\S$; in the latter case, we say that $\sigma$ \emph{holds generically} in $\S$.
\end{defn}

Informally, the zero-one law for $\L_{\omega_1, \omega}$-sentences means that generic groups from $\S$ are virtually indistinguishable.
Our first main result is the following.

\begin{thm}[Zero-one law for finitely generated groups]\label{01}
For any isomorphism-invariant closed subspace $\S\subseteq \G$, the following conditions are equivalent.
\begin{enumerate}
\item[(a)] For any non-empty open sets $U$, $V$ in $\S$, there is a finitely generated group $G$ such that $[G]\cap U\ne \emptyset$ and $[G]\cap V\ne \emptyset$.
\item[(b)] There exists a finitely generated group $G$ such that ${[G]}$ is dense in $\S$.
\item[(c)] $\S $ satisfies the zero-one law for $\mathcal L_{\omega_1, \omega}$-sentences.
\end{enumerate}
\end{thm}

In the settings of Theorem \ref{01}, one might expect that an $\L_{\omega_1,\omega}$-sentence holds generically in $\S$ if and only if it holds for a group whose isomorphism class is dense in $\S$. In general, this is false (see Example \ref{Sc}). The question of which $\L_{\omega_1,\omega}$-sentences actually have comeager sets of models in $\S$ is rather non-trivial and is discussed in more detail in Section 5.3.

We mention one applications of the zero-one law for $\mathcal L_{\omega_1, \omega}$-sentences to elementary theories, which will be used later. Recall that two groups are \emph{elementarily equivalent} if they satisfy the same first-order sentences. For any $\S\subseteq \G$, we define $Th^{gen}(\S)$ to be the set of all first-order sentences in $\L$ that hold generically in $\S$. Since $Th^{gen}(\S)$ is countable, it has a comeager set of models in $\S$. Note that for every sentence $\sigma$, we have
$$
\ModS(\sigma)\cup \ModS(\lnot \sigma)=\S.
$$
If $\S$ is a Baire space (for instance, a closed subspace of $\G$) satisfying the zero-one law for $\L_{\omega_1, \omega}$-sentences, then either $\sigma$ or $\lnot \sigma$ must hold generically in $\S$ since $\S$ cannot be covered by two meager sets. It follows that $Th^{gen}(\S)$ is complete. Since all models of a complete theory are elementarily equivalent, we obtain the following.

\begin{prop}\label{Th-gen}
A closed subspace $\S\subseteq \G$ satisfying the zero-one law for $\mathcal L_{\omega_1, \omega}$-sentences contains a comeager elementary equivalence class.
\end{prop}

Conditions (a)--(c) from Theorem \ref{01} seem rather restrictive. Our next goal is to provide non-trivial examples of subspaces $\S\subseteq \G$ for which these conditions hold. We begin by exploring a straightforward approach.

\begin{defn}
We say that a finitely generated group $G$ is \emph{condensed} if $[G]$ is non-discrete.
\end{defn}

It is easy to show that for every finitely generated group $G$, the isomorphism class $[G]$ is either discrete or has no isolated points (Corollary \ref{dych}). In the former case, the zero-one law for $\mathcal L_{\omega_1, \omega}$-sentences is vacuously true for $\S=\overline{[G]}$ as $[G]$ is comeager in $\S$. On the other hand, condensed groups lead to non-trivial instances of the zero-one law. This motivates the study of condensed groups carried out in Section 6.1. We briefly summarize our results here.

First, we show that condensed groups cannot occur in ``reasonably good" classes.

\begin{prop}\label{no-cond}
No finitely generated linear group or finitely presented residually finite group is condensed. In particular, groups of polynomial growth cannot be condensed.
\end{prop}

On the other hand, we have the following.

\begin{ex}\label{excond}
\begin{enumerate}
\item[(a)] If $G$ is a finitely generated group such that $G\cong G\times G$, then $G$ is condensed. Examples of finitely generated groups isomorphic to their direct squares were first constructed by Jones \cite{Jon}. Subsequently, Meier \cite{Mei} found a much simpler construction.
\item[(b)] The iterated monodromy group $IMG(z^2+i)$ of the polynomial $z^2+i$ is condensed. This result follows immediately from the work of Nekrashevich \cite{Nek}. It is worth mentioning that, in addition to being condensed, the group $IMG(z^2+i)$ enjoys many other interesting properties, e.g., it is residually finite, torsion, and of intermediate growth.
\end{enumerate}
\end{ex}

Interestingly, the zero-one law for $\L_{\omega_1,\omega}$-sentences applied to the space $\overline{[IMG(z^2+i)]}$ allows us to recover another result of Nekrashevich: \emph{there exist $2^{\aleph_0}$ pairwise non-isomorphic, finitely generated, residually finite groups of intermediate growth with isomorphic  profinite completions} (see Proposition \ref{Nek}).

There is a close relation between the existence of condensed groups in a subspace $\S\subseteq \G$ and the complexity of the isomorphism relation. Recall that an equivalence relation $E$ on a topological space $X$ is called \emph{smooth} if there is a Polish space $P$ and a Borel map $\beta \colon X\to P$ such that for any $x,y\in X$, we have $xEy$ if and only if $\beta(x)=\beta(y)$. Informally, smoothness means that elements of $X$ can be ``explicitly classified" up to the equivalence relation $E$ using invariants from a Polish space. The following proposition is an easy corollary of the fundamental result in the theory of Borel equivalence relations, known as the Glimm-Effros dichotomy \cite{HKL}.

\begin{prop}\label{smooth}
An isomorphism-invariant closed subset $\S\subseteq \G$ contains a condensed group if and only if the isomorphism relation on $\S$ is not smooth.
\end{prop}

We mention one immediate application. Williams \cite{W} proved that the isomorphism relation on the subspace $\S\subseteq \G$ consisting of $3$-step solvable groups is non-smooth. Thus, we obtain the following.

\begin{cor}
There exist finitely generated, condensed, solvable of step $3$ groups.
\end{cor}
It is worth noting that this result is improved in the forthcoming paper \cite{Osi20}, where (nilpotent of step $2$)-by-abelian condensed groups are constructed by a different method.  On the other hand, it is easy to show that finitely generated solvable of step $2$ groups cannot be condensed (see Proposition \ref{cond}).

We now turn to more natural examples, which served as the main source of motivation for our work. Recall that a geodesic metric space $X$ is \emph{hyperbolic} if there exists a constant $\delta\ge 0$ such that for every geodesic triangle $\Delta $ in $X$, every side of $\Delta $ is contained in the $\delta$-neighborhood of the union of the other two sides. A group is \emph{hyperbolic}  if it acts properly and cocompactly on a hyperbolic metric space. Replacing properness with its relative analogue modulo a given collection of subgroups (called \emph{peripheral subgroups}) leads to the notion of \emph{relative hyperbolicity}.

These definitions were suggested by Gromov in \cite{Gro}. Ever since, the study of hyperbolic and relatively hyperbolic groups has been one of the main directions in geometric group theory. A further generalization, the notion of an acylindrically hyperbolic group, was  introduced in \cite{Osi16} and received considerable attention in the past years. For a survey of recent developments, we refer to \cite{Osi18}.

A hyperbolic group is called \emph{elementary} if it is virtually cyclic. A relatively hyperbolic group $G$ is \emph{elementary} if it is virtually cyclic or one of the peripheral subgroups coincides with $G$. Let $\mathcal H$ (respectively, $\mathcal {RH}$, $\mathcal {AH}$) denote the subspace of $\G$ consisting of all pairs $(G,A)$ such that $G$ is a non-elementary hyperbolic (respectively, non-elementary relatively hyperbolic, acylindrically hyperbolic) group. For a subset $\mathcal Z\subseteq \G$, we let $\mathcal Z_{tf}$ (respectively, $\mathcal Z_0$) denote the subset of all torsion free groups from $\mathcal Z$ (respectively, all groups from $\mathcal Z$ without non-trivial finite normal subgroups). Thus we have the following diagram.

\begin{equation}\label{diag}
\begin{array}{ccccc}
  {\mathcal H}_{tf} & \subset  &  {\mathcal {RH}}_{tf} & \subseteq &  {\mathcal {AH}}_{tf} \\
  \cap &   & \cap &   & \cap \\
   {\mathcal H}_0 & \subset  & {\mathcal {RH}}_0 & \subseteq &  {\mathcal {AH}}_0\\
   \cap &   & \cap &   & \cap \\
   {\mathcal H} & \subset  & {\mathcal {RH}} & \subseteq &  {\mathcal {AH}}
\end{array}
\end{equation}

We also denote by $\overline{\mathcal Z}$  the closure of a subset $\mathcal Z$ in $\G$. It turns out that $\overline{\mathcal H}$, $\overline{\mathcal {RH}}$, and $\overline{\mathcal {AH}}$ do not satisfy the zero-one law for $\mathcal L_{\omega_1, \omega}$-sentences. On the other hand, Champetier \cite{Cha} showed that $\overline{\mathcal H}_{tf}$ contains a dense isomorphism class. Moreover, the small cancellation technique developed in \cite{Ols,Osi10,Hull} allows us to easily verify condition (a) from Theorem \ref{01} for the closures of all classes in the first two rows of the diagram (\ref{diag}).

\begin{thm}\label{hyp}
The subsets shown on diagram (\ref{diag}) have no isolated points and the following hold.
\begin{enumerate}
\item[(a)] Every comeager subset of $\overline{\mathcal H}$, $\overline{\mathcal {RH}}$, or $\overline{\mathcal {AH}}$ has infinitely many elementary equivalence classes; in particular, these spaces do not satisfy the zero-one law for $\mathcal L_{\omega_1, \omega}$-sentences.
\item[(b)] The spaces $\overline{\mathcal H}_{tf}$, $\overline{\mathcal {RH}}_{tf}$, $\overline{\mathcal {AH}}_{tf}$, $\overline{\mathcal H}_{0}$, $\overline{\mathcal {RH}}_{0}$, and $\overline{\mathcal {AH}}_{0}$ satisfy the zero-one law for $\mathcal L_{\omega_1, \omega}$-sentences; in particular, they contain comeager elementary equivalence classes.
\end{enumerate}
\end{thm}

An additional motivation for our work comes from the fact that the standard tools for constructing models -- ultrapowers, omitting types, and the L\" owenheim-Skolem theorem -- are not available for finitely generated groups. There are very few explicit constructions producing elementarily equivalent finitely generated groups and none of them is capable of producing ``massive" elementary equivalence classes. This difficulty is well-illustrated by the following question, sometimes attributed to Sabbagh (see, e.g., \cite{OH}): \emph{Does there exist a complete theory in the language of groups having $2^{\aleph_0}$ finitely generated models?}

In the MathOverflow post \cite{T-MO}, Thomas noticed that the affirmative answer follows from non-smoothness of the isomorphism relation on $\G$, which was proved earlier by Thomas and Velickovich \cite{Tho}. It is worth noting that Theorem  \ref{hyp} provides natural examples of complete theories (namely, $Th^{gen}(\S)$, where $\S$ is any of the spaces listed in part (b) of the theorem) having $2^{\aleph_0}$ pairwise non-isomorphic finitely generated models. Indeed, it is well-known and easy to prove that every comeager subset of a closed Polish space without isolated points has the cardinality of the continuum. Since isomorphism classes in $\G$ are countable, the comeager elementary equivalence classes from part (b) of the theorem contain $2^{\aleph_0}$ pairwise non-isomorphic groups.

It is also interesting to compare Theorem \ref{hyp} to some known results about hyperbolic groups. Geometric methods developed in \cite{Sel} allowed Sela to give a complete description of all finitely generated groups with the same elementary theory as a given torsion-free hyperbolic group. One corollary of this description is the following.

\begin{thm}[Sela, {\cite[Proposition 7.1]{Sel}}]\label{Sela}
Torsion-free hyperbolic groups having property $FA$ of Serre are elementarily equivalent if and only if they are isomorphic.
\end{thm}

Recall that a group $G$ has property $FA$ if every action of $G$ on a simplicial tree without invertions of edges fixes a vertex. Random hyperbolic groups (i.e., random groups in the Gromov density model with $d<1/2$) are almost surely torsion free and have property $FA$ \cite{DGP,Oli}. This leads to the following ``generic rigidity" phenomenon: \emph{for random hyperbolic groups, elementary equivalence is almost surely equivalent to isomorphism.}

Closely related to hyperbolic groups, is the more general class of lacunary hyperbolic groups introduced in \cite{OOS}. A finitely generated group $G$ is \emph{lacunary hyperbolic} if at least one asymptotic cone of $G$ is an $\mathbb R$-tree; for finitely presented groups, this condition is equivalent to hyperbolicity. A more ``constructive" equivalent definition is given in Section 6.3 (see Theorem \ref{LH-def}).

Lacunary hyperbolic groups occur as generic objects in $\overline{\mathcal H}$. As such, they share many common properties with random hyperbolic groups, including property $FA$ and even much stronger property (T) of Kazhdan. Thus it is reasonable to expect that generic torsion-free lacunary hyperbolic groups exhibit a similar first-order rigidity. Contrary to these expectations, the real situation is completely opposite: we prove that \emph{generic torsion-free lacunary hyperbolic groups are elementarily equivalent}.

More precisely, let $\mathcal{LH}$ denote the subspace of all pairs $(G,A)\in \G$ such that $G$ is lacunary hyperbolic and not virtually cyclic.

\begin{thm}\label{LHgen}
The spaces $\mathcal{LH}$, $\mathcal{LH}_0$, and $\mathcal{LH}_{tf}$ are dense $G_\delta$-subsets of $\overline{\mathcal H}$, $\overline{\mathcal H}_{0}$, and $\overline{\mathcal H}_{tf}$, respectively, homeomorphic to the subspace of irrational numbers in $\mathbb R$.
\end{thm}

Combining Theorems \ref{hyp} and \ref{LHgen} we immediately obtain the following.

\begin{cor}\label{LH-cor}
The number of elementary equivalence classes in every comeager subset of $\mathcal{LH}$ is infinite. The spaces $\mathcal{LH}_{tf}$ and $\mathcal{LH}_0$ contain comeager elementary equivalence classes.
\end{cor}

\paragraph{Structure of the paper and advice to the reader.}
In order to make this paper as self-contained as possible, we review some necessary material from logic and descriptive set theory in the next section. Some of our results, namely Theorems \ref{hyp} and \ref{LHgen}, are proved using a somewhat technical machinery from geometric group theory, which is surveyed in Section 4. Readers not familiar with this technique can simply read the main definitions and examples in Subsections 4.1 and 4.2, skip Subsection 4.3, and take the lemmas proved in Subsection 4.4 as black boxes; modulo these results, all the proofs in our paper are fairly elementary and accessible to a wide mathematical audience. Theorem \ref{01} and some of its corollaries are proved in Section 5. Section 6 is devoted to examples and applications. In particular, we prove Theorem \ref{hyp} and Theorem \ref{LHgen} in Subsections 6.2 and 6.3, respectively. Some open questions related to our work are collected in Section 7. We also briefly discuss a possible generalization of our work to general $\Omega$-algebras there (see Theorem \ref{01A}).

\section{Background from topology and logic}\label{Prelim}

\paragraph{3.1. Elements of descriptive set theory.}
Recall that a topological space is \emph{Polish} if it is completely metrizable and separable. A basic example is the power set $2^S$ of a (discrete) countable set $S$ endowed with the product topology (or, equivalently, the topology of pointwise convergence of indicator functions). By the Tychonoff's theorem, $2^S$ is always compact. If $S$ is countably infinite, $2^S$ is homeomorphic to the Cantor set. The class of Polish spaces is obviously closed under taking closed subspaces.

A\emph{ $G_\delta$-subset} (respectively, an \emph{$F_\sigma$-subset}) of a topological space $X$ is a subset that can be represented as a countable intersection of open sets (respectively, a countable union of closed sets).

\begin{prop}[{\cite[Theorem 3.11]{Kech}}]\label{ps-cl}
A $G_\delta $-subspace of a Polish space is a Polish space.
\end{prop}

A subset $Y$ of a topological space $X$ is \emph{meager} if it is a union of countably many nowhere dense sets. A \emph{comeager} set is a set whose complement is meager. Equivalently, a subset is comeager if it is an intersection of countably many sets with dense interior. It follows directly from the definitions that the class of meager (respectively, comeager) subsets is closed under countable unions (respectively, countable intersections).

A topological space $X$ is a \emph{Baire space} if the intersection of any countable collection of dense open sets is dense in $X$. Equivalently, Baire spaces can be characterized by any of the following two equivalent conditions:

\begin{enumerate}
\item[(B$_1$)] every comeager set is dense;

\item[(B$_2$)] no non-empty open set is meager.
\end{enumerate}

In particular, comeager sets are non-empty in a non-empty Baire space. The result below is known as the \emph{Baire Category Theorem}, see~\cite[Theorem~8.4]{Kech}.

\begin{thm}\label{BCT}
Every completely metrizable space is Baire. In particular, every Polish space is Baire.
\end{thm}

Summarizing the above discussion, we can regard meager and comeager sets of Polish spaces as ``negligible" and ``generic", respectively.

Recall that an action of a group $G$ on a topological space $X$ is \emph{topologically transitive} if for every non-empty open sets $U,V\subseteq X$, there exists $g\in G$ such that $g(U)\cap V\ne \emptyset$. Our proof of the zero-one law for finitely generated groups will make use of the well-known relation between topological transitivity and existence of dense orbits for homeomorphism groups of Polish spaces. Since there is some ambiguity in the literature on this topic (in fact, the result is false for general topological spaces), we provide a brief proof.

\begin{lem}[Folklore]\label{tt}
Suppose that a group $G$ acts topologically transitively by homeomorphisms on a Polish space $X$. Then the set
$$
\{ x\in X\mid \overline{Gx}=X\}
$$
is comeager in $X$. In particular, there exists a dense orbit.
\end{lem}

\begin{proof}
Since $X$ is separable and metrizable, there exists a countable basis of neighborhoods $\{U_i\}_{i\in \mathbb N}$. Let
$$
C=\bigcap_{i\in \NN}\bigcup_{g\in G} gU_i.
$$
Since the action of $G$ is topologically transitive, $C$ is comeager in $X$. Let $x\in C$ and let $V\subseteq X$ be a non-empty open set. We have $U_i\subseteq V$ for some $i$. Since $x\in \bigcup_{g\in G} gU_i$, there exists $g\in G$ such that $g^{-1}x\in U_i \subseteq V$. Thus $Gx$ is dense in $X$.
\end{proof}

Yet another ingredient of the proof of Theorem \ref{01} is the \emph{topological zero-one law}. The first result of this sort was proved by Oxtoby \cite{Oxt} as early as in 1937. It then underwent a sequence of generalizations culminating in the following.

\begin{thm}[ {\cite[Theorem 8.46]{Kech}}]\label{01top}
Suppose that a group $G$ acts topologically transitively by homeomorphisms on a Baire space $X$. Then every $G$-invariant Borel subset of $X$ is either meager or comeager.
\end{thm}

\begin{proof}
The proof is fairly elementary and we provide a sketch for convenience of the reader. Recall that a subset $A$ of any topological space $X$ has the \emph{Baire property} (abbreviated BP) if it can be decomposed as $A=U\vartriangle M$, where $U$ is open and $M$ is meager. It is easy to show that the set of all subsets of $X$ having the BP is a $\sigma$-algebra, see \cite[Proposition 8.22]{Kech}. It follows that every Borel set has the BP.

Suppose that $A$ is a $G$-invariant Borel subset of $X$. Then we have $A=U\vartriangle M$ and $X\setminus A=V\vartriangle N$, where $U$, $V$ are open and $M$, $N$ are meager in $X$. Assume first that both $U$ and $V$ are non-empty. By topological transitivity of the action, there is $g\in G$ such that $W=g(U)\cap V\ne \emptyset$. Using the assumption that $A$ is $G$-invariant, it is easy to show that $W\subseteq g(M)\cup N$. Since $g$ is a homeomorphism, $W$ is open and meager, which is impossible in a Baire space (see (B$_2$) above). This contradiction shows that one of the sets $U$, $V$ must be empty, which means that $A$ is either meager or comeager.
\end{proof}

\paragraph{3.2. Infinitary logic.}
In the infinitary logic $\L_{\omega_1, \omega}$, terms, atomic formulas, and formulas are constructed from the symbols of $\mathcal L$ using the same syntactic rules as in the standard first-order logic with one exception: countably infinite conjunctions and disjunctions of formulas are allowed. For the formal definitions, see \cite[Chapter 1]{Mar16}. Throughout this paper, a ``formula" means an $\mathcal L_{\omega_1, \omega}$-formula and the term ``first-order formula" is used to distinguish finitary formulas, i.e., those formulas which do not involve conjunctions or disjunctions over infinite sets. To help the reader become familiar with $\L_{\omega_1, \omega}$, we consider several examples.

Let $G$ be a group and let $w_1, w_2, ...$ be an enumeration of the set of all words in the alphabet $\{ x_1, x_1^{-1}, \ldots, x_n, x_n^{-1}\}$; we think of the words $w_i$ as terms in the language of groups. Consider the formula
\begin{equation}\label{alpha}
\alpha (x_1, \ldots, x_n) = \forall\, g\, \left(\bigvee_{i\in \NN} g=w_i \right).
\end{equation}
It is easy to see that $\alpha (x_1, \ldots, x_n)$ holds for a subset $\{x_1, \ldots , x_n\}\subseteq G$  if and only if $G=\langle x_1, \ldots, x_n\rangle$. In particular, we have
$$
G\models \bigvee_{n\in \NN} (\exists\, x_1\; \ldots\; \exists\, x_n \; \alpha (x_1, \ldots, x_n))
$$
if and only if $G$ is finitely generated. Note that the property of being finitely generated is not definable in the first-order logic since any theory with infinite models must have uncountable models by the L\"owenheim-Skolem theorem.

Similarly, one can show that all the classes of groups listed below can be defined by $\L_{\omega_1, \omega}$-sentences while none of them is first-order definable.

\begin{enumerate}
\item[(a)] Torsion groups, simple groups, and solvable groups.
\item[(b)] Amenable groups.
\item[(c)] Groups having property (T) of Kazhdan.
\item[(d)] Any countable set of isomorphism classes of groups (e.g., finitely presented groups).
\end{enumerate}

For classes listed in (a) the proof is straightforward. Definability of the class of amenable groups follows from the observation that the F\o lner criterion can be written as an $\L_{\omega_1, \omega}$-sentence. For (c), one has to use a result of Shalom \cite{Sha}: every property (T) group is a quotient of a finitely presented property (T) group. Finally, (d) can be derived from the following proposition, which is a particular case of a more general result of Scott \cite{Sco}.

\begin{thm}[Scott]\label{ScTh}
For every finitely generated group $G$, there exists an $\L_{\omega_1,\omega}$-sentence $\sigma$ with the following property: for any group $H$, we have $H \models \sigma$ if and only if $H\cong G$.
\end{thm}

\begin{proof}
Let $a_1, \ldots, a_n $ be a generating set of $G$. Let $u_1(x_1, \ldots, x_n), u_2(x_1, \ldots, x_n), \ldots $ (respectively, $v_1(x_1, \ldots, x_n), v_2(x_1, \ldots, x_n), \ldots $) be an enumeration of all words in the alphabet $\{ x_1, x^{-1}, \ldots, x_n^{-1}, x_n\}$, which we think of as terms in the language of groups, such that $u_i (a_1, \ldots,a_n)=1$ (respectively, $v_i(a_1, \ldots , a_n)\ne 1$) in $G$ for all $i$. Then the formula
$$
\exists\, x_1\; \ldots\; \exists\, x_n \; \left(\alpha (x_1, \ldots, x_n) \wedge \left(\bigwedge_{i\in \NN} \big(u_i(x_1, \ldots, x_n)=1  \wedge  (\lnot v_i(x_1, \ldots, x_n)= 1)\big)\right)\right),
$$
where $\alpha (x_1, \ldots, x_n)$ is defined by (\ref{alpha}), has the required property.
\end{proof}

The examples considered above show that the expressive power of $\L_{\omega_1, \omega}$ is much higher than that of the ordinary first order logic. In fact, it is difficult to find an example of a natural group theoretic property that \emph{cannot} be expressed by a $\L_{\omega_1, \omega}$-sentence. One such an example, due to Wesolek and Williams, is the property of being elementary amenable; for more examples and the proof, see \cite{WW}.

We now discuss complexity classes of first-order and $\L_{\omega_1, \omega}$-sentences. Recall that every first-order formula is a equivalent to a formula in the \emph{prenex normal form}, where all the quantifiers are moved to the front. The following classes of first-order formulas will be of particular importance. A formula is called \emph{universal} (respectively, \emph{existential}) if its prenex normal form only involves universal (respectively, existential) quantifiers. A \emph{$\forall\exists$-formula} is a formula whose prenex normal form has a string of universal quantifiers, followed by a string of existential quantifiers, followed by a quantifier-free formula. Similarly, we define $\exists\forall$-formulas and so on.

Unlike in the first-order logic, there is no prenex normal form in $\mathcal L_{\omega_1, \omega}$. However, every formula is equivalent to a formula in the \emph{negation normal form}, where negation only applies to atomic formulas. Formulas in the negation normal form can be divided into \emph{complexity classes} $\Sigma_\alpha$ and $\Pi_\alpha$, where $\alpha $ is a countable ordinal as follows.

\begin{enumerate}
\item[(a)] Every quantifier-free first-order formula is in $\Sigma_0=\Pi_0$.

\item[(b)] For every countable $\alpha>0$, let $\Sigma _\alpha$ denote the smallest set of $\mathcal L_{\omega_1, \omega}$-formulas that contains $\bigcup_{\beta<\alpha} \Pi_\beta$ and is closed under countable disjunctions, finite conjunctions, and adding existential quantifiers.

\item[(c)] Similarly, for every $\alpha >0$, let $\Pi _\alpha$ denote the smallest set of $\mathcal L_{\omega_1, \omega}$-formulas that contains $\bigcup_{\beta<\alpha} \Sigma_\beta$ and is closed under countable conjunctions, finite disjunctions, and adding universal quantifiers.
\end{enumerate}

Here we say that a class of formulas $\Theta$ is \emph{closed under adding universal (respectively, existential) quantifiers} if for every formula $\theta (x)\in \Theta$ with a free variable $x$, we have $\forall\, x\; \theta(x)\in \Theta$ (respectively, $\exists\, x\; \theta(x)\in \Theta$).

In particular, every universal (respectively, existential) first-order formula is in $\Pi_1$ (respectively, $\Sigma_1$), every $\forall\exists$-formula is in the class $\Pi_2$, and so on.

\paragraph{3.3. The space of finitely generated groups.}
Our next goal is to review various approaches to topologizing the space of finitely generated groups. Recall that $\G_n$ denotes the set of equivalence classes of pairs $(G,A)$, where $G$ is a group and $A\subseteq G^n$ is an ordered generating set of $G$ (i.e., an $n$-tuple whose elements generate $G$), modulo the following equivalence relation: $$(G, (a_1, \ldots, a_n))\approx (H, (b_1, \ldots , b_n))$$ if the map $a_1\mapsto b_1,\; \ldots,\; a_n\mapsto b_n$ extends to an isomorphism $G\to H$. To simplify our notation, we write $(G,A)$ for the $\approx$-class of $(G,A)$.

In this paper, we often think of graphs as metric spaces. More precisely, we equip each graph $\Gamma$ with the length metric induced by identifying open edges with the interval $(0,1)$. Thus, for a graph $\Gamma $ and a vertex $v$ of $\Gamma$, the \emph{ball of radius $r$ around $v$} consists of all vertices within the distance at most $r$ from $v$ and all edges $e$ such that at least one end of $e$ is within the distance at most $(r-1)$ from $v$; if $r=0$, the ball consists of $v$ only.

We say that $(G,A), (H,B)\in \G_n$, where $A=(a_1, \ldots, a_n)$ and $B=(b_1, \ldots, b_n)$,  are \emph{$r$-similar} for some $r\in \mathbb N$ and write $(G,A)\approx_r (H,B)$ if there is an isomorphism (in the category of directed graphs) between the balls of radius $r$ around the identity in the Cayley graphs $\Gamma (G,A)$ and $\Gamma (H,B)$ that takes edges labeled by $a_i$ to edges labeled by $b_i$ for all $i$. The topology on $\G_n$ is defined by taking the sets
\begin{equation}\label{UGA}
U_{G,A}(r)=\{ (H,B)\in \G_n\mid (H,B)\approx _r (G,A)\},
\end{equation}
where $(G,A)$ ranges in $\G$ and $r\in \NN$, as the base of neighborhoods. Thus a sequence $\{(G_i, A_i)\}$ converges to $(G,A)$ in $\G_n$ if for every $r\in \mathbb N$, $(G_i,A_i)$ and $(G,A)$ are $r$-similar for all sufficiently large $i$. It is easy to see that $\approx$ is the intersection of the equivalence relations $\approx_r$ and hence the topology on $\G_n$ is well-defined.

\begin{ex}\label{Zex}
It is easy to see that $(\ZZ/m\ZZ, \{ 1\}) \approx_r (\ZZ, \{ 1\})$ for $r=\lfloor m/2\rfloor -1$. In particular, we have $\lim\limits_{m\to\infty} (\ZZ/m\ZZ, \{ 1\}) =(\ZZ, \{ 1\}).$
\end{ex}

The definition of the spaces $\G_n$ is due to Grigorchuk \cite{Gri}. It is customary to stack all spaces $\G_n$ together as follows.   Identifying $(G,(a_1, \ldots, a_n))$ with $(G,(a_1, \ldots, a_n, 1))$ we get an inclusion of $\G_n$ into $\G_{n+1}$ as a clopen subset. The topological union $$\G=\bigcup_{n\in \NN} \G_n$$ is called the \emph{space of marked finitely generated groups}. Thus, a sequence $\{(G_i, A_i)\mid i\in \NN\}$ converges to $(G,A)$ in $\G$ if and only if $\{(G_i, A_i)\mid i\in \NN\}\subseteq \G_n$ for some $n\in \NN$ and $\{(G_i, A_i)\mid i\in \NN\}$ converges to $(G,A)$ in $\G_n$.

It is useful to extend the relation $\approx_r$ from the individual spaces $\G_n$ to the whole $\G$ as follows: we say that $(G,A)\approx _r (H,B)$ for some $(G,A), (H,B)\in \G$ if $(G,A)$ and $(H,B)$ are $r$-similar as elements of some (equivalently, any) $\G_n$ to which they both belong. Further, we can define $\d\colon \G\times \G \to \G$ by letting 
\begin{equation}\label{Eq:defdG}
\d ((G,A), (H,B))=\frac1{1+{\max \{ r\in \NN\cup\{ 0\} \mid (G,A)\approx_r (H,B)\}}}.
\end{equation}
It is easy to see that $\d$ is a metric and the corresponding metric space is a metrization of $\G$.

Grigorchuk \cite{Gri} proved that every $\G_n$ equipped with the restriction of $\d$ is a Hausdorff, second countable, zero-dimensional, compact space (recall that a topological space is \emph{zero-dimensional} if it has a basis of neighborhoods consisting of clopen sets). In particular, each $\G_n$ is Polish. By the definition of $\d$, for any $n\in \NN$, any $(G,A)\in \G_n$, and any $(H,B)\in \G\setminus \G_n$, we have $\d((G,A), (H,B))=1$. Therefore, every Cauchy sequence in $(\G, \d)$ entirely belongs to some $\G_n$. This easily implies that $(\G, \d)$ is complete. Since second countability of $\G$ is obvious, we obtain the following well-known fact.

\begin{prop}[{Grigorchuk}]\label{Gri}
$\G$ is a Hausdorff, second countable, zero-dimensional, $\sigma$-compact, Polish space.
\end{prop}

Recall that a topological space is \emph{$\sigma$-compact} if it is a countable union of compact subspaces.

\begin{conv}
We say that a marked group $(G,A)\in \G$ has some group theoretic property, if so does $G$.
\end{conv}

In some cases, it is useful to think of the spaces $\G_n$ in a different way. Let $F_\infty=F(x_1, x_2, \ldots)$ denote the free group with countably infinite basis $\{x_1,x_2, \ldots\}$ and let
$$
\mathcal N(F_\infty) =\{ N\lhd F_\infty \mid x_i\in N \; {\rm for \; all\; but\; finitely\; many\;} i\}.
$$
Recall that $2^{F_\infty}$ denotes the space of all subsets of $F_\infty$ with the product topology. We think of $\mathcal N(F_\infty)$ as a subset of $2^{F_\infty}$ and endow it with the induced topology. A similar space (namely, the space of all closed subgroups of a locally compact topological group) was studied by Chabauty \cite{Ch}; for this reason, the product topology on $\mathcal N(F_\infty)$ is sometimes called the \emph{Chabauty topology}.

Every $(G,A)\in \G$ can be naturally identified with an element of $\mathcal N(F_\infty)$. Namely, $(G,(a_1, \ldots, a_n))$ corresponds to the kernel of the natural homomorphism $\e_{G,A}\colon F_\infty\to G$ such that
\begin{equation}\label{eGA}
\e_{G,A} (x_i)=\left\{\begin{array}{cl}
                 a_i, & {\rm for}\; i=1, \ldots, n, \\
                 1, & {\rm if }\; i>n.
               \end{array}\right.
\end{equation}
It is easy to prove the following.

\begin{prop}\label{GriCha}
The map $(G,A)\mapsto Ker\, \e_{G,A}$ defines a bijective continuous map $f\colon \mathcal G\to \mathcal N(F_\infty)$. The restriction of this map to $\G_n$ yields a homeomorphism $\G_n$ to $\mathcal N_n(F_\infty)$, where $\mathcal N_n(F_\infty)$ consists of all $N\lhd F_\infty$ such that $x_i\in N$ for all $i>n$.
\end{prop}

It is worth noting that the map $f\mathcal G\to \mathcal N(F_\infty)$ defined above is not a homeomorphism, as was erroneously claimed in the previous version of this paper. Indeed, let $G=\ZZ=\langle x\rangle$ and let $A_i=(x, 1, \ldots , 1, x)\in G^i$, $i\ge 3$. The sequence of marked groups $(G, A_i)$ does not converge in $\mathcal G$ while their images in $\mathcal N(F_\infty)$ converge to the image of $(G, \{ x\})$. Moreover, it is not difficult to show $\mathcal N(F_\infty)$ is not a Polish space (although all spaces $\mathcal N_n(F_\infty)$ are).

We consider a couple of standard examples. The proofs are straightforward using either the geometric definition of $\G$ or one of the alternative spaces $\mathcal N_n(F_\infty)$.

\begin{ex} \label{exconv} Let $G$ be a finitely generated group, $A$ a finite generating set of $G$, $N$ a normal subgroup of $G$.
\begin{enumerate}
\item[(a)] Suppose that $N_1\ge N_2\ge \cdots $ is a sequence of normal subgroups of $G$ such that $N=\bigcap_{i\in \NN} N_i$. Let $Q_i=G/N_i$, $Q=G/N$, and let $X_i$ (respectively, $X$) be the natural image of $A$ in $Q_i$ (respectively, $Q$). Then $\lim_{i\to \infty} (Q_i, X_i)=(Q,X)$ in $\G$. This can be seen as a generalization of Example \ref{Zex}.
\item[(b)] The same claim holds for any increasing sequence of normal subgroups $N_1\le N_2\le \ldots$ such that $N=\bigcup_{i\in \NN} N_i$. In particular, every marked group in $\G$ is a limit of finitely presented marked groups.
\end{enumerate}
\end{ex}

The following useful result also goes back to \cite{Gri}.
\begin{prop}[Grigorchuk,  \cite{Gri}]\label{Quot}
Let $(G,A)\in \G$. If $G$ is finitely presented, then there is a neighborhood $U$ of $(G,A)$ in $\G$ such that for every $(H,B)\in U$, the group $H$ is a quotient of $G$.
\end{prop}

\begin{proof}
Let $G=\langle A\mid \mathcal R\rangle $ be a finite presentation and let $r$ denote the maximum of the lengths of relations in $\mathcal R$. Then every finitely generated group $H$ satisfying $(H,B)\approx_{\max\{ 1, r\}} (G,A)$ is a quotient of $G$.
\end{proof}

One reason for working with $\G$ instead of individual spaces $\G_n$ is that the isomorphism classes in $\G$ are the orbits of a natural group action. More precisely, let $\Aut $ denote the group of finitary automorphisms of $F_\infty$. That is, an automorphism $\alpha\in Aut(F_\infty)$ belongs to $\Aut $ if and only if $\alpha(x_i)=x_i$ for all but finitely many $i$. The group $\Aut $ acts on $\mathcal N(F_\infty)$, which gives rise to an action on $\G$ via the one-to-one correspondence between $\mathcal N(F_\infty)$ and $\G$ established in Proposition \ref{GriCha}. Note that, for every $\alpha \in \Aut$ and every $n\in \NN$, there is $k\in \NN$ such that $\alpha (\mathcal N_n(F_\infty))\subseteq \mathcal N_k(F_\infty)$. This and the fact that the restriction of the map $f$ from Proposition \ref{GriCha} to each $\G_k$ yields a homeomorphism $\G_k\to \mathcal N_k(F_\infty)$, implies that the induced action of $\Aut $ on $\G$ is continuous. Alternatively, one can define the same action of $\Aut$ on $\G$ using elementary Nielsen transformations and verify its continuity directly.

The following observation is due to Thomas \cite{Tho09}.

\begin{lem}[Thomas]\label{Aut}
$\Aut $ acts on $\G$ by homeomorphisms. Isomorphism classes in $\G$ coincide with orbits of the $\Aut $-action.
\end{lem}

\begin{proof}
The proof of the first claim is straightforward. The second claim follows from the fact that for every $(G, (a_1, \ldots, a_m))\in \G$, an $n$-tuple $(b_1, \ldots, b_n)\subseteq G^n$ generates $G$ if and only if the $(m+n)$-tuple $(a_1, \ldots, a_m, 1, \ldots, 1)$ can be transformed to the $(m+n)$-tuple $(b_1, \ldots, b_n, 1, \ldots, 1)$ by a finite sequence of elementary Nielsen transformations.
\end{proof}

Finally we note that the space $\G$ can be identified with a subspace of the space of countably infinite structures traditionally studied in model theory. More precisely, let $\mathcal R$ be a first-order language whose signature consists of countably many predicates $R_1, R_2, \ldots $. We denote by $a(R_i)$ the arity of $R_i$ and consider the set
$$
\mathcal X(\mathcal R) =2^{\mathbb N ^{a(R_1)}}\times 2^{\NN^{a(R_2)}} \times \cdots ,
$$
endowed with the product topology. Every  $M\in \mathcal X(\mathcal R)$ defines an $\mathcal R$-structure whose universe is $\NN$ and the interpretation of $R_i$ is the $a_i$-ary relation on $\NN$ given by the $i$-th component of $M$. The space $\mathcal X(\mathcal R)$ is called the \emph{space of countable $\mathcal R$-structures} \cite[Section 3.1]{Mar16}.

The requirement that the signature of $\mathcal R$ consists of predicates only is not really restrictive. Indeed, every structure can be turned into a relational structure by replacing all operations with their graphs. Moreover, for every marked group $(G,A)$ (including finite groups), we can construct the corresponding relational structure $S(G,A)$ as follows.

Let $\mathcal R$ denote the first order language with signature $\{ P_1,P_2,P_3\}$, where $P_i$ is an $i$-ary predicate for each $i=1,2,3$.
We fix an enumeration $F_\infty=\{f_1, f_2, \ldots\}$ and identify $\NN$ with $F_\infty$ via the map $i\mapsto f_i$. Given $(G,A)\in \G$, the universe of the structure $S(G,A)$ is $F_\infty=\NN$ and the predicate symbols are interpreted as follows. For all $u,v,w\in F_\infty$, we have:
$$
P_1(w)=True {\rm \;\;if \; and\; only\; if\;\;} \e_{G,A}(w)=1,
$$
$$
P_2(u,v)=True {\rm \;\;if \; and\; only\; if\;\;}  \e_{G,A}(v)=\e_{G,A}(u^{-1}),
$$
$$
P_3(u,v,w)=True {\rm \;\;if \; and\; only\; if\;\;} \e_{G,A}(uv)=\e_{G,A}(w),
$$
where the homomorphism $\e_{G,A}\colon F_\infty \to G$ is defined by (\ref{eGA}).

\begin{prop}\label{GXR}
In the notation introduced above, the following statements hold.
\begin{enumerate}
\item[(a)] The map $(G,A)\to S(G,A)$ defines a continuous embedding $\G \to \mathcal X(\mathcal R)$.

\item[(b)] There is a rewriting procedure $\sigma_{\L}\mapsto \sigma_{\mathcal R}$ that transforms $\L_{\omega_1, \omega}$-sentences to sentences in the corresponding infinitary version of $\mathcal R$ such that for every finitely generated group $G$, we have $G\models \sigma_{\L}$ if and only if $S(G,A)\models \sigma_{\mathcal R}$ for all finite generating sets $A$ of $G$.
\end{enumerate}
\end{prop}

This result, together with the classical L\'opez-Escobar theorem \cite{LE}, can be used to derive Proposition \ref{Borel}, which is crucial for the proof of the zero-one law for finitely generated groups. We will use an alternative approach and give a self-contained proof of Proposition \ref{Borel}. Thus Proposition \ref{GXR} will not be used in our paper and so we leave the (straightforward) proof to the reader.

\section{Hyperbolic groups and their generalizations}

\paragraph{4.1. Hyperbolic groups.}
Throughout this paper, we often think of graphs as metric spaces. Given a connected graph $\Gamma$, we identify every edge of $\Gamma$ with $[0,1]$ and define the distance between two points $a,b\in \Gamma$ to be the length of the shortest path in $\Gamma$ connecting $a$ to $b$. Note that this distance is defined for all points in $\Gamma$, not necessarily vertices.

We begin by recalling Gromov's definition of hyperbolicity of a metric space \cite{Gro}. Recall that a metric space $S$ with a distance function $\d$ is called \emph{geodesic}, if every two points $a,b\in S$ can be connected by a path $p$ of length $\ell(p)=\d(a,b)$.

\begin{defn}
Let $\delta $ be a non-negative constant. A metric space $S$ is \emph{$\delta$-hyperbolic} if it is geodesic and for any geodesic triangle $\Delta $ in $S$, every side of $\Delta $ is contained in the union of the closed $\delta$-neighborhoods of the other two sides. A metric space is called \emph{hyperbolic} if it is $\delta$-hyperbolic for some $\delta \ge 0$.
\end{defn}
We list some classical examples.

\begin{ex}\label{exhs}
\begin{enumerate}
\item[(a)] Every bounded geodesic space $S$ is hyperbolic with $\delta ={\rm diam} (S)$.
\item[(b)] Every tree is hyperbolic with $\delta=0$.
\item[(c)] $\mathbb H^n$ is hyperbolic for every $n\in \mathbb N$ while $\mathbb R^n$ is not hyperbolic for $n\ge 2$.
\item[(d)] In the class of geodesic metric spaces, hyperbolicity is preserved by the quasi-isometry relation, see \cite[Chapter III.H, Theorem 1.9]{BH} for details.
\end{enumerate}
\end{ex}

\begin{defn}
A group $G$ is \emph{hyperbolic} if it is finitely generated and for some (equivalently, any) generating set $A$, the Cayley graph $\Gamma (G,A)$ is hyperbolic. A hyperbolic group is called \emph{elementary} if it has a cyclic subgroup of finite index and \emph{non-elementary} otherwise.
\end{defn}

Example \ref{exhs} can be translated into the following.

\begin{ex}\label{exhyp}
\begin{enumerate}
\item[(a)] Every finite group is hyperbolic.
\item[(b)] Every finitely generated free group is hyperbolic.
\item[(c)] If $M$ is a closed hypebolic manifold, then $\pi_1(M)$ is hyperbolic.
\item[(d)] The class of hyperbolic groups is closed under passing to finite extensions, subgroups of finite index, extensions with finite kernel, and quotients by finite normal subgroups.
\end{enumerate}
\end{ex}

An important class of examples consists of small cancellation groups. Recall that a word $w$ is the alphabet $X\cup X^{-1}=\{x_1, x_1^{-1}, x_2, x_2^{-1}, \ldots \} $ is \emph{reduced} if it contains no subwords of the form $x_ix_i^{-1}$ and $x_i^{-1}x_i$ and \emph{cyclically reduced} if every cyclic shift of $w$ is reduced. When talking about group presentations, we always assume that relations are cyclically reduced. A group presentation
\begin{equation}\label{eq-pres}
G=\langle X \mid \mathcal R\rangle.
\end{equation}
is said to be \emph{symmetrized} if for every word $R\in \mathcal R$, all cyclic shifts of $R^{\pm 1}$ belong to $\mathcal R$. If the presentation (\ref{eq-pres}) is not symmetrized, by its \emph{symmetrization} we mean the presentation obtained by adding all cyclic shifts of $R^{\pm 1}$ for all $R\in \mathcal R$ to the set of relations. A symmetrized presentation (\ref{eq-pres}) satisfies the $C^\prime(\lambda)$ \emph{small cancellation condition} for some $\lambda \in [0, 1]$ if any common initial subword $U$ of two distinct words $R,S\in \mathcal R$ satisfies
$$
\| U\| < \lambda \min\{ \| R\|, \, \| S\|\},
$$
where $\| \cdot \| $ denotes the number of letters in the corresponding word. A non-symmetrized group presentation satisfies $C^\prime(\lambda)$ if so does its symmetrization.

The standard example is the presentation of the fundamental group of a closed surface of genus $g$,
$$
\langle a_1, b_1, \ldots, a_n, b_n\mid a_1b_1a_1^{-1}b_1^{-1}\cdots a_nb_na_n^{-1}b_n^{-1}\rangle,
$$
which satisfies $C^\prime(\lambda)$ for every $\lambda > \frac{1}{4n}$.

The result below was originally proved by Greendlinger \cite{Green} using combinatorial arguments. For a contemporary geometric proof see \cite[Section 12]{Ols-book}.

\begin{lem}[Greendlinger Lemma]\label{Green}
Suppose that a group $G$ has a symmetrized presentation (\ref{eq-pres}) satisfying $C^\prime (1/6)$. Then every non-empty reduced word in the alphabet $X\cup X^{-1}$ that represents $1$ in $G$ contains a subword $U$ which is also a subword of some $R\in \mathcal R$ such that $\| U\| > \|R\|/2$.
\end{lem}

It is known that a group is hyperbolic if and only if it has a finite presentation (\ref{eq-pres}) satisfying the conclusion of the Greendlinger lemma \cite[Chapter III.$\Gamma$, Theorem 2.6]{BH}. In particular, we have the following result (for a direct proof using geometry of triangles, see \cite{Str}).

\begin{cor}\label{c16}
Every group given by a finite presentation satisfying $C^\prime (1/6)$ is hyperbolic.
\end{cor}

It is worth noting that the constant $1/6$ is optimal. That is, Lemma \ref{Green} and Corollary \ref{c16} generally fail for group presentations satisfying $C^\prime(\lambda )$ for $\lambda >1/6$. Moreover, the condition $C^\prime (\lambda)$ becomes essentially unrestrictive for $\lambda >1/5$, see \cite{Gol} or Section 12.4 of \cite{Ols-book}.

We consider a particular example, which will be used (together with the Greendlinger lemma) in the proofs of Theorem \ref{hyp} and Theorem \ref{LHgen}.

\begin{ex}\label{Wnk}
For any $k,n\in \NN$, let
$$
W_n(k)=\langle u, v \mid R_1, \ldots , R_n\rangle .
$$
where
\begin{equation}\label{wj}
R_j=u^jvu^jv^{2}\ldots u^jv^{k}
\end{equation}
for all $j=1, \ldots, n$. It is easy to see that for every sufficiently large $k$ (e.g., $k\ge 30$) and every $n\in \NN$, this presentation satisfies the $C^\prime(1/6)$ condition and hence $W_n(k)$ is hyperbolic.
\end{ex}

\paragraph{4.2. Relatively and acylindrically hyperbolic groups.}
The notion of relative hyperbolicity was also suggested by Gromov in \cite{Gro}. Since then many equivalent definitions have been formulated. For a comprehensive survey, we refer to \cite{Hru}. We recall the definition based on the notion of a hyperbolically embedded collection of subgroups introduced in \cite{DGO}, which is most suitable for our purpose.

Let $G$ be a group, $\Hk$ a collection of subgroups of $G$, $X$ a (possibly infinite) subset of $G$. Suppose that
$$
G=\langle X\cup H_1\cup \ldots \cup H_k\rangle .
$$
We denote by $\Gxh $ the Cayley graph of $G$ whose edges are labeled by letters from the alphabet
$$
\mathcal A=X\sqcup H_1\sqcup \ldots \sqcup H_k.
$$
That is, two vertices $f,g\in V(\Gamma (G, \mathcal A))=G$ are connected by an edge going from $f$ to $g$ and labeled by $a\in \mathcal A$ iff $fa=g$ in $G$. Disjointness of the union in the definition of $\mathcal A$ means that if a letter $a\in H_i$ and a letter $b\in X$ (or $b\in H_j$ for $j\ne i$) represent the same element of $G$, then for every $f\in G$, the Cayley graph $\Gxh $ will have two edges connecting $f$ and $fa=fb$: one labeled by $a$ and the other labeled by $b$.

For each $i$, we naturally think of the Cayley graph $\Gamma_{H_i}=\Gamma (H_i,H_i)$ of $H_i$ with respect to the generating set $H_i$ as a (complete) subgraph of $\Gxh $. In this notation, we have the following.

\begin{defn}\label{he-def}
A collection of subgroups $\Hk$ is \emph{hyperbolically embedded in $G$ with respect to $X$}, denoted $\Hk \h (G,X)$, if the following conditions hold.
\begin{enumerate}
\item[(a)] The Cayley graph $\Gxh$ is hyperbolic.
\item[(b)] For every $i\in \{ 1, \ldots, k\}$ and every $n\in \mathbb N$, there are only finitely many elements $h\in H_i$ such that the vertices $h$ and $1$ can be connected in $\Gxh$ by a path of length at most $n$ avoiding edges of $\Gamma _{H_i}$.
\end{enumerate}
Further, we say  that the collection $\Hk$ is \emph{hyperbolically embedded} in $G$ and write $\Hk\h G$ if $\Hk\h (G,X)$ for some $X\subseteq G$.
\end{defn}

To illustrate this definition, we consider standard examples borrowed from \cite{DGO}.

\begin{ex}\label{bex}
\begin{enumerate}
\item[(a)] For any group $G$ we have $G\h G$.  We can take $X=\emptyset $ in this case. Similarly, if $H$ is a finite subgroup of a group $G$, then we always have $H\h G$. Indeed $H\h (G,G)$.

\item[(b)] Let $G=H\times \mathbb Z$, $X=\{ x\} $, where $x$ is a generator of $\mathbb Z$, $\mathcal A=X\sqcup H$. Then $\Gamma (G,\mathcal A)$ is quasi-isometric to a line and hence it is hyperbolic. However, every two elements $h_1, h_2\in H$ can be connected by a path of length at most $3$ in $\Gamma(G, \mathcal A)$ that avoids edges of $\Gamma _H$ (see Fig. \ref{fig0}). Thus $H\not\h (G,X)$ whenever $H$ is infinite.

\item[(c)]  Let $G=H\ast \mathbb Z$, $X=\{ x\} $, where $x$ is a generator of $\mathbb Z$, $\mathcal A=X\sqcup H$. In this case $\Gamma (G, \mathcal A=X\sqcup H)$ is quasi-isometric to a tree and no path connecting $h_1, h_2\in H$ and avoiding edges of $\Gamma_H$ exists unless $h_1=h_2$. Thus $H\h (G,X)$.
\end{enumerate}
\end{ex}

\begin{figure}
  \centering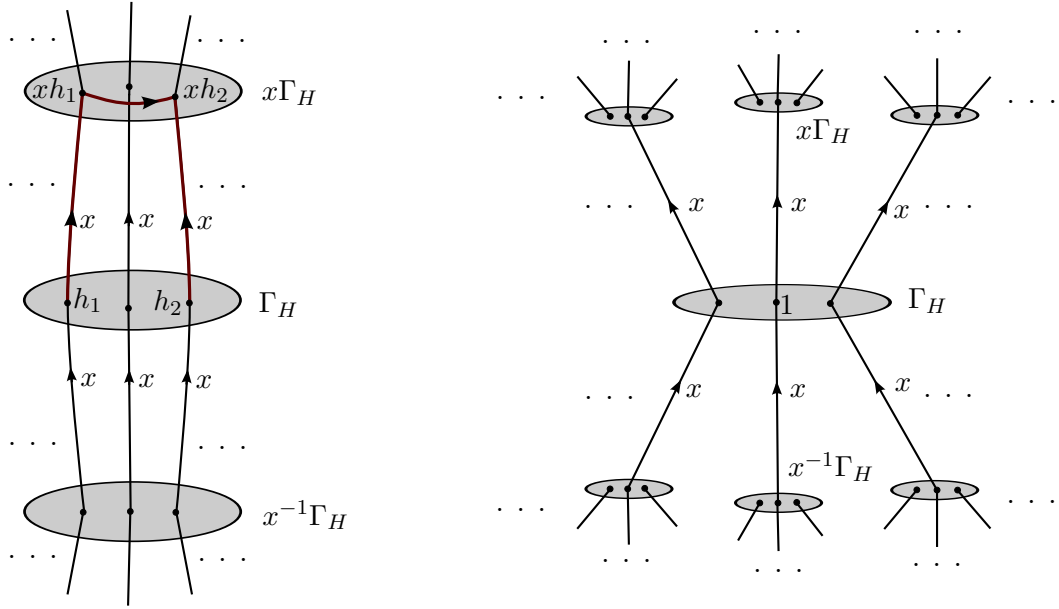
  \caption{Cayley graphs $\Gamma(G, \mathcal A)$ for $G=H\times \mathbb Z$ and $G=H\ast \mathbb Z$.}\label{fig0}
\end{figure}

We list some useful properties of hyperbolically embedded subgroups. The proposition below can easily be proved by generalizing the idea behind Example \ref{bex} (b).

\begin{prop}[Dahmani-Guirardel-Osin, {\cite[Proposition 4.33]{DGO}}]\label{malnorm}
Let $G$ be a group, $\Hk$ a collection of hyperbolically embedded subgroups of $G$
\begin{enumerate}
\item[(a)] For all $i$ and all $g\in G\setminus H$, we have $|g^{-1}H_ig\cap H_i|<\infty $.
\item[(b)] For all $i\ne j$ and all $g\in G$, we have $|g^{-1}H_ig\cap H_j|<\infty$.
\end{enumerate}
\end{prop}

\begin{cor}\label{K=1}
Suppose that a group $G$ contains a non-trivial torsion-free hyperbolically embedded subgroup. Then $G$ has no non-trivial finite normal subgroups.
\end{cor}

\begin{proof}
Suppose that $K$ is a finite normal subgroup of $G$.  Then the centralizer $C_{G} (K)$ has finite index in $G$. For every $k\in K$, we have $|k^{-1} Hk\cap H|\ge |C_G(K)\cap H|=\infty $. By Proposition \ref{malnorm} (a), this implies $K\le H$. Since $H$ is torsion-free, we have $K=\{ 1\}$.
\end{proof}

Recall that for a group $G=\langle A\rangle$, the word length of an element $g\in G$ with respect to $A$, denoted $|g|_A$, is the minimal $n$ such that $g=a_1\ldots a_n$, where $a_i\in A\cup A^{-1}$ for all $i$. A finitely generated subgroup $H$ of a group $G$ is \emph{undistorted} in $G$ with respect to some (not necessarily finite) generating set $A$ of $G$ if for some (equivalently, any) finite generating set $B$ of $H$, there exists a constant $c>0$ such that
$$
|h|_B\le c|h|_A
$$
for all $h\in H$. The next result is obtained in the course of proving Theorem 4.31 in \cite{DGO} (see inequality (35) there). See also \cite[Corollary 4.8]{AHO} for a much stronger result.

\begin{prop}[Dahmani-Guirardel-Osin, \cite{DGO}]\label{undist}
Let $G$ be a finitely generated group, $A$ a (possibly infinite) generating set of $G$, $H$ a subgroup of $G$. Suppose that $H\h (G,A)$. Then $H$ is finitely generated and undistorted in $G$ with respect to $A$.
\end{prop}

We now give the definition of a relatively hyperbolic group based on hyperbolically embedded subgroups. For finitely generated groups it is equivalent to all other definitions, see \cite[Proposition 4.28]{DGO} and \cite[Corollary 2.48]{Osi06}.

\begin{defn}\label{rhdef}
A group $G$ is hyperbolic relative to a collection of subgroups $\Hk$, called  \emph{peripheral subgroups}, if and only if $\Hk\h (G,X)$ for some finite subset $X\subseteq G$. We say that $G$ is \emph{non-elementary relatively hyperbolic} if it is not virtually cyclic and is hyperbolic relative to a collection of proper subgroups.
\end{defn}

\begin{rem}\label{RHne}
By \cite[Lemma 4.4]{Osi06b}, if $G$ is hyperbolic relative to $\Hk$ and one of the subgroups $H_i$ is proper and infinite, then all other subgroups are proper.
\end{rem}

\begin{ex}\label{RHex}
\begin{enumerate}
\item[(a)] Every group is hyperbolic relative to itself.
\item[(b)] A group is hyperbolic if and only if it is hyperbolic relative to the trivial subgroup (or, equivalently, with respect to the empty collection of subgroups).
\item[(c)] If $G=H_1\ast H_2$, then $G$ is hyperbolic relative to $\{ H_1, H_2\}$.
\item[(d)] Let $M$ be a compact orientable 3-manifold with connected boundary such that  $M\setminus\partial M$ admits a complete finite volume hyperbolic structure. Then $\pi_1(M)$ is hyperbolic relative to $\pi_1(\partial M)$ \cite{Gro}.
\end{enumerate}
\end{ex}

We mention one useful property of relative hyperbolicity.

\begin{prop} [{\cite[Corollary 2.41]{Osi06}}]\label{rhh}
If a finitely generated group $G$ is hyperbolic relative to a collection of hyperbolic subgroups, then $G$ is hyperbolic itself.
\end{prop}

Finally, we discuss acylindrically hyperbolic groups. An isometric action of a group $G$ on a metric space $S$ is {\it acylindrical} if for every $\e>0$ there exist $R,N>0$ such that for every two points $x,y\in S$ with $\d (x,y)\ge R$, there are at most $N$ elements $g\in G$ satisfying
$$
\d(x,gx)\le \e \;\;\; {\rm and}\;\;\; \d(y,gy) \le \e.
$$
The notion of acylindricity goes back to Sela's paper \cite{Sel97}, where it was considered for groups acting on trees. In the context of general metric spaces, the above definition is due to Bowditch \cite{Bow}.

Informally, one can think of this condition as a kind of properness of the action on $S\times S$ minus a ``thick diagonal".

\begin{ex}
\begin{enumerate}
\item[(a)] It is easy to see that proper cocompact actions are acylindrical.
\item[(b)] Every group action on a bounded space $S$ is acylindrical. Indeed, it suffices to take $R>{\rm diam}(S)$.
\end{enumerate}
\end{ex}

Every group has an acylindrical action on a hyperbolic space, namely the trivial action on the point. For this reason, we want to avoid elementary actions in the definition below. Recall that an action of a group $G$ on a hyperbolic space $S$ is \emph{non-elementary} if the limit set of $G$ on the Gromov boundary $\partial S$ has infinitely many points. For readers unfamiliar with the notions of the Gromov boundary and limit sets, we mention a useful  equivalent characterization: an acylindrical action of a group $G$ on a hyperbolic spaces is non-elementary if and only if $G$ is not virtually cyclic and the action has infinite orbits {\cite[Theorem 1.1]{Osi16}}.

\begin{thm}[{\cite[Theorem 1.2]{Osi16}}]\label{aa}
For any group $G$, the following conditions are equivalent.
\begin{enumerate}
\item[(AH$_1$)] $G$ admits a non-elementary acylindrical action on a hyperbolic space.
\item[(AH$_2$)] $G$ is not virtually cyclic and there exists a (possibly infinite) generating set $A$ of $G$ such that the corresponding Cayley graph $\Gamma (G,A)$ is hyperbolic, unbounded, and the natural action of $G$ on $\Gamma (G,A)$ is acylindrical.
\item[(AH$_3$)] $G$ contains a proper infinite hyperbolically embedded subgroup.
\end{enumerate}
\end{thm}

\begin{defn}\label{ahdef}
A group $G$ is \emph{acylindrically hyperbolic} if it satisfies either of the equivalent conditions  (AH$_1$)--(AH$_3$).
\end{defn}

The class of acylindrically hyperbolic groups includes non-elementary hyperbolic and relatively hyperbolic groups, mapping class groups of closed surfaces of non-zero genus, $Out(F_n)$ for $n\ge 2$, non-virtually cyclic groups acting properly on proper $CAT(0)$ spaces and containing a rank-$1$ element, groups of deficiency at least $2$, most $3$-manifold groups, automorphism groups of some algebras (e.g., the Cremona group of birational transformations of the complex projective plane) and many other examples. For more details we refer to the survey \cite{Osi18}.

By \cite[Theorem 2.24]{DGO}, every acylindrically hyperbolic group (in particular, every non-elementary hyperbolic or non-elementary relatively hyperbolic group) contains a unique maximal finite normal subgroup denoted by $K(G)$.  We call $K(G)$ the \emph{finite radical} of $G$.  In what follows, we will need a slightly more precise version of condition (AH$_3$) obtained in \cite{ABO} (a similar result is also proved in \cite{DGO}). By $F_2$ we denote the free group of rank $2$.

\begin{prop}[Abbott-Balasubramanya-Osin, {\cite[Proposition 5.13]{ABO}}]\label{KG}
Let $G$ be an acylindrically hyperbolic group. Suppose that $A$ is a generating set of $G$ satisfying condition (AH$_2$). Then $G$ contains a subgroup $H\cong K(G)\times F_2$ such that $H\h (G,A)$.
\end{prop}

We will also use the following ``relatively hyperbolic version" of Proposition \ref{KG}.

\begin{cor}\label{F2RH}
Let $G$ be a finitely generated group hyperbolic relative to a collection of subgroups $\Hk$. Suppose that $G$ is not virtually cyclic and $H_i\ne G$ for  all $i$. Then there exists a subgroup  $H\cong K(G)\times F_2$ such that $G$ is hyperbolic relative to $\{H, H_1, \ldots, H_k\}$.
\end{cor}

\begin{proof}
By definition, $\Hk\h (G,X)$ for some finite set $X$. In particular, the Cayley graph $\Gamma (G,A)$ is hyperbolic, where $A=X\cup H_1\cup\ldots\cup H_k$. Under the assumptions of the corollary, $\Gamma (G,A)$ is known to be unbounded \cite[Corollary 4.6]{Osi06b} and the action of $G$ on $\Gamma (G, A)$ is known to be acylindrical \cite[Proposition 5.2]{Osi16}. Thus the generating set $A$ satisfies (AH$_2$). By Proposition \ref{KG}, there exists $H\h (G,A)$ such that $H\cong K(G)\times F_2$. The subgroup $H$ is undistorted in $G$ with respect to $A$ by Proposition \ref{undist} and satisfies $|g^{-1}Hg \cap H|<\infty $ for all $g\in G\setminus H$ by Proposition \ref{malnorm}. By the main result of \cite{Osi06b}, these properties imply that $G$ is hyperbolic relative to $\{H, H_1, \ldots, H_k\}$.
\end{proof}

\paragraph{4.3. Group theoretic Dehn filling.}
Our next goal is to discuss a group theoretic analogue of Thurston's theory of hyperbolic Dehn surgery. It will be used to show that the  subspaces of $\G$ shown on diagram (\ref{diag}) have no isolated points and that $\mathcal{LH}$, $\mathcal{LH}_0$, and $\mathcal{LH}_{tf}$ are homeomorphic to the subspace of irrational numbers in $\mathbb R$. Readers not interested in these results can safely skip Sections 4.3 and 4.4.

The classical Dehn surgery on a $3$-dimensional manifold consists of cutting of a solid torus, which may be thought of as ``drilling" along an embedded knot, and then gluing it back in a different way. The study of such transformations is partially motivated by the  Lickorish-Wallace theorem, which states that every closed orientable connected $3$-manifold can be obtained from the $3$-dimensional sphere by performing finitely many surgeries.

The second part of the surgery, called {\it Dehn filling}, can be formalized as follows. Let $M$ be a compact orientable 3-manifold with toric boundary. Topologically distinct ways of attaching a solid torus to $\partial M$ are parameterized by free homotopy classes of unoriented essential simple closed curves in $\partial M$, called {\it slopes}. For a slope $s$, the corresponding   Dehn filling  $M(s )$ of $M$ is the manifold obtained from $M$ by attaching a solid torus $\mathbb D^2\times \mathbb S^1$ to $\partial M$ so that the meridian
$\partial \mathbb D^2$ goes to a simple closed curve of the slope $s$. The fundamental theorem due to Thurston \cite[Theorem 1.6]{Th} asserts that if $M\setminus\partial M$ admits a complete finite volume hyperbolic structure, then $M(s)$ is hyperbolic for all but finitely many slopes.

Given a group $G$ and a subset $M\subseteq G$, we denote by $\ll M\rr $ the minimal normal subgroup of $G$ containing $M$. In view of Example \ref{RHex} (d) and Proposition \ref{rhh}, the theorem below can be regarded as an algebraic counterpart of Thurston's result.

\begin{thm}[Dahmani-Guirardel-Osin, {\cite[Theorem 2.27]{DGO}}]\label{CEP}
Let $G$ be a group, $\Hk$ a collection of subgroups of $G$. Suppose that $\Hk\h (G,X)$ for some $X\subseteq G$. Then for every finite subset $\mathcal K\subseteq G$, there exists a finite subset $\mathcal F=\mathcal F(\mathcal K)\subseteq G\setminus\{ 1\}$ such that for any subgroups $N_i\lhd H_i$ satisfying $N_i\cap \mathcal F=\emptyset$, the following statements hold.
\begin{enumerate}
\item[(a)] For every $i$, the natural map from $\widehat{H_i}=H_i/N_i$ to $\widehat{G}= G/\ll N_1\cup \ldots \cup N_k\rr $ is injective (equivalently, we have $H_i\cap \ll N_1\cup \ldots \cup N_k\rr =N_i$). In what follows, we think of $\widehat{H_i}$ as subgroups of $\widehat{G}$.

\item[(b)] $\{\widehat{H_1}, \ldots , \widehat{H_k}\} \h (\widehat{G}, \widehat {X})$, where $\widehat X$ is the natural image of $X$ in $\widehat{G}$. In particular, if $G$ is hyperbolic relative to $\Hk$, then $\widehat G$ is hyperbolic relative to $\{\widehat{H_1}, \ldots , \widehat{H_k}\}$.

\item[(c)] The natural homomorphism $G\to \widehat{G} $ is injective on $\mathcal K$.

\item[(d)] If $G$ and all $\widehat{H_i}$ are torsion-free, then so is $\widehat{G}$.
\end{enumerate}
\end{thm}

For relatively hyperbolic groups, this theorem was proved by the author in \cite{Osi07}; the particular case of torsion-free groups was independently obtained by Groves and Manning \cite{GM}. The generalization to hyperbolically embedded subgroups is given in \cite{DGO}.

\paragraph{4.4. Quotients of hyperbolic-like groups.}
We now prove two lemmas, which will be used later in the proofs of Theorem \ref{hyp} and \ref{LHgen}. Both lemmas are certainly implicit in the literature, and several particular cases have been already stated explicitly. The proofs are simple applications of small cancellation theory in groups acting on hyperbolic spaces. For hyperbolic groups, it  was developed by Olshanskii \cite{Ols} (following an insight by Gromov \cite{Gro} and the earlier work of Olshanskii on the geometric solution of the Burnside problem \cite{Ols82, Ols-book}) and then generalized to relatively and acylindrically groups by the author \cite{Osi10} and Hull \cite{Hull}. Readers unfamiliar with the methods of  \cite{Ols} and their generalizations in \cite{Hull,Osi10} can accept the results of this section as a ``black box".

By abuse of notation, we write $G\in \mathcal Z$ for a group $G$ and a subset $\mathcal Z\subseteq \G$ if $(G,A)\in \mathcal Z$ for some generating set $A$.

\begin{lem}\label{lemdense}
Let $\mathcal Z$ be one of the sets $ {\mathcal{H}}$, $ {\mathcal{H}_0}$, $ {\mathcal{H}_{tf}}$, $ {\mathcal{RH}}$, $ {\mathcal{RH}_0}$, $ {\mathcal{RH}_{tf}}$, $ {\mathcal{AH}}$, $ {\mathcal{AH}_0}$, $ {\mathcal{AH}_{tf}}$, and let $G\in \mathcal Z$. For every finite subset $\mathcal K$ of $G$, there exists a group $Q\in \mathcal Z$ and a non-injective epimorphism $\e\colon G\to Q$ such that the restriction $\e\vert_{\mathcal K}$ is injective.
\end{lem}

For the spaces $ {\mathcal{H}_0}$, $ {\mathcal{H}_{tf}}$, $ {\mathcal{RH}_0}$, $ {\mathcal{RH}_{tf}}$, $ {\mathcal{AH}_0}$, and ${\mathcal{AH}_{tf}}$, this is an immediate consequence of the results of \cite{Ols,Osi10,Hull} on small cancellation theory in the respective classes of groups. In fact, for every fixed  $\mathcal K$, the quotient of $G$ obtained by adding a single ``random" relation (in any conceivable sense) will satisfy the required property almost surely.

In the case of groups with non-trivial finite normal subgroups, we have to be a bit more careful. Indeed, adding ``random" relations to $G$ may kill $K(G)$ with non-zero probability. However, the required quotient $Q$ can still be obtained by adding certain special relations. E.g., if $G\in \mathcal H$, then for any element $g$ of infinite order, there exists $m\in \NN$ such that $G/\ll g^m\rr$ is  non-elementary hyperbolic and the natural map $G\to G/\ll g^m\rr$ is injective on $\mathcal K$ \cite[Theorem 3]{Ols}. Unfortunately, the analogue of this theorem for $ {\mathcal{RH}}$  and $ {\mathcal{AH}}$ does not seem to have been stated explicitly in the literature, although its proof is straightforward using either small cancellation or group theoretic Dehn filling technique. For expository purpose, we provide a proof of Lemma \ref{lemdense} based on Dehn fillings, which works in all cases.

\begin{proof}
Let $G\in \mathcal{AH}$ and let $\mathcal K$ be a finite subset of $G$. By Proposition \ref{KG}, $G$ contains a hyperbolically embedded subgroup $H\cong F_2\times K(G)$. Let $\{ x,y\}$ be a basis of $F_2$ and let $\mathcal F=\mathcal F(\mathcal K)\subseteq G\setminus\{ 1\}$ be the finite set provided by Theorem \ref{CEP} applied to $G$ and $H$. It is easy to show that there is a non-trivial element $w\in F_2\le H$ such that the normal closure $N$ of $w$ in $H$ (which coincides with the normal closure of $w$ in $F_2$) avoids $\mathcal F$, and $F_2/N$ is torsion-free and non-elementary hyperbolic. E.g., we can take the element $$w=xyxy^2\cdots xy^{k}$$ for a sufficiently large $k$, see Example \ref{Wnk} and Lemma \ref{Green}. Note that $H/N \cong F_2/N \times K(G)$ is also non-elementary hyperbolic (see Example \ref{exhyp} (d)). By Theorem \ref{CEP}, the map $G\to \widehat{G}=G/\ll N\rr$ is injective on $\mathcal K$ and the group $\widehat H = H/N$ is hyperbolically embedded in $\widehat G$. If $\widehat H=\widehat{G}$, then $\widehat{G}$ is non-elementary hyperbolic and hence acylindrically hyperbolic. If $\widehat H\ne \widehat{G}$, then $\widehat{G}$ is acylindrically hyperbolic by definition (see condition (AH$_3$) in Theorem \ref{aa}).

If $G\in \mathcal {AH}_0$, then $K(G)=\{ 1\}$ and $\widehat H$ is a torsion-free group. Therefore, $\widehat G\in \mathcal {AH}_0$ by Corollary \ref{K=1}. Note also that if $G\in \mathcal{AH}_{tf}$ then $\widehat G\in \mathcal{AH}_{tf}$ by part (d) of Theorem \ref{CEP}.

Further, suppose that $G\in \mathcal{RH}$. Let $\Hk$ be a collection of proper subgroups of $G$ such that $G$ is hyperbolic relative to $\Hk$. By Corollary \ref{F2RH}, we can find $H\cong F_2\times K(G)$ such that $G$ is hyperbolic relative to $\{ H, H_1, \ldots, H_k\}$. That is,
\begin{equation}\label{HH}
\{ H, H_1, \ldots, H_k\}\h (G,X)
\end{equation}
for some finite subset $X\subseteq G$. We then choose $w$ as above so that Theorem \ref{CEP} applies to the hyperbolically embedded collection (\ref{HH}). Consider the filling of $G$ corresponding to the collection of subgroups $N=\ll w\rr \lhd H$ and $N_i=\{ 1\}\lhd H_i$. By Theorem \ref{CEP}, the resulting group $\widehat G$ is hyperbolic relative to $\{ H/N, H_1, \ldots, H_k\}$ in this case and $G\to \widehat{G}=G/\ll N\rr$ is injective on $\mathcal K$. The fact that $G$ is non-elementary relatively hyperbolic is obvious if $\widehat H=\widehat G$ and follows from Remark \ref{RHne} if $\widehat H\ne \widehat G$. In the cases $G\in \mathcal{RH}_0$ and  $G\in \mathcal{RH}_{tf}$, we argue as in the previous paragraph.

Finally, let $G\in \mathcal H$. We can think of this as a particular case of the situation considered in the previous paragraph with empty collection $\Hk$ (see Example \ref{RHex} (b)). In this case, $\widehat G$ is hyperbolic relative to $\widehat H$ and we have $G\in \mathcal H$ by Proposition \ref{rhh}. The cases $G\in \mathcal{H}_0$ and  $G\in \mathcal{H}_{tf}$ are treated as above.
\end{proof}

\begin{cor}\label{dense}
The spaces $ {\mathcal{H}}$, $ {\mathcal{H}_0}$, $ {\mathcal{H}_{tf}}$, $ {\mathcal{RH}}$, $ {\mathcal{RH}_0}$, $ {\mathcal{RH}_{tf}}$, $ {\mathcal{AH}}$, $ {\mathcal{AH}_0}$, and $ {\mathcal{AH}_{tf}}$ have no isolated points.
\end{cor}

\begin{proof}
Let $\mathcal Z$ be one of these spaces, let $(G,A)\in \mathcal Z$, and let $U$ be any neighborhood of $(G,A)$ in $\G$. By the definition of the topology on $\G$, there exists $r\in \NN$ such that every $(H,B)\in \G$ satisfying $(H,B)\approx_r(G,A)$ belongs to $U$. Let $$\mathcal F=\{ g\in G \mid |g|_A\le r\},$$ where $|\cdot |_A$ denotes the word length with respect to $A$, and let $Q\in \mathcal Z$ be the group provided by Lemma \ref{lemdense}. It is straightforward to see that $(Q, \e(A))\approx _r (G,A)$ and hence $(Q, \e(A))\in U$.
\end{proof}

\begin{lem}\label{Q}
Let $\mathcal Z$ be one of the sets  $ {\mathcal{H}_0}$, $ {\mathcal{H}_{tf}}$, $ {\mathcal{RH}_0}$, $ {\mathcal{RH}_{tf}}$, $ {\mathcal{AH}_0}$, or $ {\mathcal{AH}_{tf}}$. For any groups $G_1,G_2\in \mathcal Z$ and any finite subsets $\mathcal F_i\subseteq G_i$, there exist a group $Q\in \mathcal Z$ and epimorphisms $\e_i\colon G_i\to Q$ such that the restriction of $\e_i$ to $\mathcal F_i$ is injective for $i=1,2$.
\end{lem}

This result is also well-known to experts. In the case $\mathcal Z=\mathcal{H}_{tf}$, this lemma was proved by Champetier \cite{Cha}. The case $\mathcal Z=\mathcal{AH}_0$ is done in \cite[Corollary 7.4]{Hull}. For other classes, weaker versions (without the injectivity condition) were proved in \cite{Ols00} and \cite{AMO}. In fact, the injectivity condition follows from the same proofs; the only reason it was not stated explicitly in \cite{Ols00} and \cite{AMO} is that it was unnecessary for the applications considered there. We provide a brief proof for completeness.

\begin{proof}
We begin with the case $\mathcal Z= {\mathcal{H}_0}$. Let $G_1$, $G_2$ and $\mathcal F_1$, $\mathcal F_2$ be as above and let $P=G_1\ast G_2$. Then $P$ is hyperbolic and $G_1$ and $G_2$ do not normalize any non-trivial finite subgroup of $P$. The latter condition means that both $G_1$ and $G_2$ are Gromov subgroups of $P$ in the terminology of \cite{Ols}. By \cite[Theorem 2]{Ols}, there exists a non-elementary hyperbolic group $Q$ and an epimorphism $\e\colon P\to Q$ such that $\e\vert_{\mathcal F_1\cup \mathcal F_2}$ is injective and $\e(G_1)=\e(G_2)=Q$; in addition, we can ensure that $Q$ has no non-trivial finite normal subgroups by part 9) of \cite[Theorem 1]{Min}. Thus $Q\in \mathcal{H}_0$.  Dealing with $\mathcal Z= {\mathcal{H}_{tf}}$ is similar. In this case, part (7) of \cite[Theorem 2]{Ols}  guarantees that the group $Q$ can be made torsion-free and so the reference to \cite[Theorem 1]{Min} is unnecessary.

For $ {\mathcal{RH}_0}$, $ {\mathcal{RH}_{tf}}$, $ {\mathcal{AH}_0}$, and $ {\mathcal{AH}_{tf}}$, the proof is identical modulo the changes indicated below. The notion of a Gromov subgroup must be replaced with the more restrictive notion of a suitable subgroup, see \cite{AMO, Osi10} and \cite{Hull} for the definition in the relatively hyperbolic and acylindrically hyperbolic case, respectively. The reference to Olshanskii's theorem \cite[Theorem 2]{Ols} must be replaced with the references to \cite[Theorem 2.4]{Osi10} and \cite[Theorem 1.5]{Hull}, respectively. Finally, the reference to \cite[Theorem 1]{Min} for classes $\mathcal{RH}_0$ and $\mathcal {AH}_0$ can be eliminated. Indeed, \cite[Theorem 2.4]{Osi10} and \cite[Theorem 1.5]{Hull} imply that the obtained common quotient group $Q$ contains a suitable subgroup; in turn, this implies that $Q$ has no non-trivial finite normal subgroups (see \cite[Proposition 3.4]{AMO} for the relatively hyperbolic case and \cite[Definition 1.4]{Hull} for acylindrically hyperbolic groups).
\end{proof}

\section{The zero-one law for $\L_{\omega_1, \omega}$-sentences}

\paragraph{5.1. Complexity of $\L_{\omega_1, \omega}$-sentences and Borel hierarchy.}
Our proof of the zero-one law for finitely generated groups makes use of the observation that an isomorphism-invariant subset of $\G$ is Borel if and only if it is definable by an $\L_{\omega_1, \omega}$-sentence. This has long been known to experts and can be derived from the L\'opez-Escobar theorem for the space of countable structures (see \cite{LE} or \cite[Section 3.1]{Mar16}) via Proposition \ref{GXR}. For expository purpose, we choose a different approach and provide a direct proof (which is by no means original). In fact, we obtain a somewhat more precise result; to state it, we need to recall the definition of the Borel hierarchy.

Let $X$ be a topological space. For every countable ordinal $\alpha \ge 1$, we define the classes $\mathbf \Sigma_\alpha^0$ and $\mathbf \Pi_\alpha^0$ of Borel subsets of $X$ as follows.

\begin{enumerate}
\item[(a)] The class $\mathbf \Sigma_1^0$ consists of open sets.
\item[(b)] For $\alpha \ge 1$, a set is in $\mathbf \Pi_\alpha ^0$ if its complement is in $\mathbf \Sigma_\alpha^0$.
\item[(c)] For $\alpha > 1$, a set is in $\mathbf \Sigma_\alpha^0$ if it is a countable union of sets from $\bigcup_{\beta<\alpha} \mathbf \Pi^0_\beta$.
\end{enumerate}
In particular, $\mathbf \Pi_1^0$ consists of closed sets, $\mathbf \Sigma_2^0$ consists of $F_\sigma$-sets, and $\mathbf \Pi_2^0$ consists of $G_\delta$-sets. An easy transfinite induction shows that $\mathbf\Sigma_\alpha^0$ is closed under finite intersections (and countable unions) and $\mathbf\Pi_\alpha^0$ is closed under finite unions (and, of course, countable intersections).

\begin{prop}\label{Borel}
For every countable ordinal $\alpha \ge 1$, the following statements hold.
\begin{enumerate}
\item[(a)] Suppose that $\sigma $ is an $\L_{\omega_1,\omega}$-sentence of complexity class $\Sigma_\alpha$ or $\Pi_\alpha$. Then $\ModG(\sigma)$ is an isomorphism-invariant Borel subset of $\G$ of class $\mathbf\Sigma_\alpha^0$ or  $\mathbf \Pi_\alpha^0$ respectively.
\item[(b)] Conversely, for any isomorphism-invariant Borel set $X\subseteq \G$ of class $\mathbf\Sigma_{\alpha}^0$ or $\mathbf \Pi_{\alpha}^0$, there is an $\L_{\omega_1,\omega}$-sentence $\sigma$ such that $X={\rm Mod}_{\G} (\sigma)$.
\end{enumerate}
\end{prop}

\begin{rem}
Before proceeding with the proof, we note that it is not, in general, true that a subset $X\subseteq \G$ of Borel class $\mathbf \Sigma_\alpha^0$ or $\mathbf \Pi_\alpha^0$ is definable by an $\L_{\omega_1, \omega}$-sentence of complexity $\Sigma_\alpha$ (respectively, $\Pi_\alpha$). For example, let $K$ be any non-trivial finite group. The set $X=[K]$ is closed but cannot be defined by a $\Pi_1$-sentence as the set of models of every $\Pi_1$-sentence is closed under taking subgroups while $X$ is not. Similarly, it is easy to show that the set of all $2$-generated groups is open but cannot be defined by a $\Sigma_1$-sentence.
\end{rem}

\begin{proof}
Consider the language $\mathcal L^\prime$ obtained from $\mathcal L$ by adding constants $c_1, c_2, \ldots$. Given $(G,A)\in \G$, where $A=(a_1, \ldots, a_n)$, we interpret $c_i$ as $a_i$ for $1\le i\le n$ and as $1$ for $i>n$. Thus every $(G,A)\in \G$ becomes an $\mathcal L^\prime$-structure. We also denote by $\mathcal L^\prime _{\omega_1, \omega}$ the corresponding infinitary version of $\mathcal L^\prime$.

Let $T$ denote the set of all $\L^\prime $-terms that only involve constants $c_1, c_2, \ldots$. That is, $T$ is the set of all words in the alphabet $c_1, c_1^{-1}, c_2, c_2^{-1},\ldots $.  For every $\mathcal L^\prime_{\omega_1, \omega}$-formula $\phi (x)$ with one free variable $x$ and for every $(G,A)\in \G$, we obviously have
\begin{equation}\label{eq1}
(G,A) \models \forall x\, \phi(x) \;\; {\rm if\; and\; only\; if} \;\; (G,A) \models \bigwedge_{t\in T} \phi (t)
\end{equation}
and
\begin{equation}\label{eq2}
(G,A) \models \exists x\, \phi(x) \;\; {\rm if\; and\; only\; if}\;\; (G,A) \models \bigvee_{t\in T} \phi (t)
\end{equation}
since $G$ is generated by $A$.

We first prove the following ``non-invariant version" of Proposition \ref{Borel}. For convenience, we denote by $\mathbf\Sigma_0^0= \mathbf\Pi_0^0$ the set of all clopen subsets of $\G$. 

\begin{lem} \label{claim}
Let $\alpha \ge 0$ be a countable ordinal. Suppose that $\sigma $ is an $\L^\prime_{\omega_1,\omega}$-sentence of complexity $\Sigma_\alpha$ (respectively, $\Pi_\alpha$). Then ${\rm Mod}_{\G} (\sigma)$ ia a Borel subset of $\G$ and is in the class $\mathbf\Sigma_\alpha^0$ (respectively, $\mathbf \Pi_\alpha^0$). Conversely, for every Borel subset $X\subseteq \G$, there exists an $\L^\prime_{\omega_1,\omega}$-sentence $\sigma$ such that $X={\rm Mod}_{\G} (\sigma)$.
\end{lem}

\begin{proof}
To prove the first claim, we proceed by induction on $\alpha$. Let $\sigma $ be an $\L^\prime_{\omega_1,\omega}$-sentence. We note that $\sigma \in \Sigma_0=\Pi_0$ if and only if $\sigma$ is a Boolean combination of finitely many atomic formulas of the form $t_1=t_2$ for some $t_1, t_2\in T$. It is easy to see that $\ModG (\sigma)$ is clopen this case. Assume now that $\alpha >0$ and $\sigma \in \Pi _\alpha$. By the definition of $\Pi _\alpha$, $\sigma $ is obtained from formulas in the set $\bigcup_{\beta <\alpha}\Sigma _\beta$ by taking countable conjunctions, finite disjunctions,  and adding universal quantifiers. By (\ref{eq1}) and the inductive assumption, we have $\ModG (\sigma)\in \mathbf \Pi_\alpha^0$. Similarly, we can do the inductive step for $\sigma \in \Sigma _\alpha$ using (\ref{eq2}). 

To prove the second claim, we first recall that, in a separable (equivalently, second countable) metric space, every open set is a countable union of open balls. Applying this to $(\G, \d)$, where $\d$ is defined by (\ref{Eq:defdG}), we can easily derive that every open set of $G$ is a countable union of sets $U_{(G,A)}(r)=\{(H,B)\in \G \mid (H,B)\approx_r(G,A)\}$ for some $(G,A)\in \G$ and $r\in \NN$. 

Furthermore, every $U_{(G,A)}(r)$ can be defined by a countable conjunction of atomic formulas. Indeed, if $A=(a_1, \ldots, a_n)$, we can take the conjunction of all atomic formulas $t=1$ and $s\ne 1$ involving only variables $c_1, \ldots, c_n$ such that the corresponding relations and non-relations between elements $\{ a_1, \ldots, a_n\}$ can be ``seen" in the ball of radius $r$ around the identity in the Cayley graph $\Gamma(G, A)$; in addition, we add all relations of the form $c_i=1$ for $i>n$. Therefore, every open subset of $\G$ can be defined by an $\L^\prime_{\omega_1,\omega}$-sentence. The claim for every Borel subset follows from this by induction using (\ref{eq1}) and (\ref{eq2}).
\end{proof}

We now return to the proof of the proposition. Let $\sigma $ be an $\L_{\omega_1, \omega}$-formula of class $\Sigma_\alpha$ (respectively, $\Pi_\alpha$). Clearly, $\sigma$ can also be thought of as an $\L^\prime _{\omega_1, \omega}$-formula of the same complexity class and $\ModG(\sigma)$ is isomorphism-invariant. Thus part (a) of Proposition \ref{Borel} follows immediately from Lemma \ref{claim}.

To prove part (b), let $X$ be an isomorphism-invariant Borel subset of $\G$. By Lemma \ref{claim}, there is an $\L^\prime_{\omega_1, \omega}$-sentence $\sigma$ such that $X=\ModG(\sigma)$. We  denote by $\sigma_n$ the $\L_{\omega_1, \omega}$--formula obtained from $\sigma $ by replacing all the occurrences of the constants $c_i$ with the variables $x_i$ for $i\le n$ and with the constant $1$ for $i>n$. Further, let
$\xi_n= \exists\, x_1\;\ldots\; \exists\, x_n\; (\alpha_n\wedge\sigma_n)$,
where $\alpha_n=\alpha (x_1, \ldots, x_n)$ is the formula defined by (\ref{alpha}). Recall that $\alpha_n$ means that $\{ x_1, \ldots, x_n\}$ is a generating set. For every $(G,A)\in X$ we have $(G,A)\in \G_n$ for some $n$. Therefore, $G\models \xi_n$. This shows that $X\subseteq \ModG(\xi)$ for $\xi=\bigvee_{n\in\NN}\xi_n$. Conversely, if $G\models \xi$, then $G$ satisfies $\xi_n$ for some $n$. That is, there exists a generating set $A$ of size $n$ such that $(G,A)\in \ModG(\sigma)=X$. Thus $X= \ModG(\xi)$. 
\end{proof}

\paragraph{5.2. Proof of Theorem \ref{01}.}
For some applications, it is useful to reformulate the zero-one law in terms of a ``generalized elementary equivalence".  Let $F$ be a theory in $\L_{\omega_1, \omega}$ (i.e., $F$ is simply a set of $\L_{\omega_1,\omega}$-sentences). We say that two groups $G$ and $H$ are \emph{$F$-equivalent} if $G$ and $H$ satisfy exactly the same sentences from $F$. In particular, taking $F$ to be the first-order logic we obtain the definition of the standard elementary equivalence.

We are now ready to prove a result incorporating both Theorem \ref{01} and Proposition~\ref{Th-gen}.

\begin{thm}\label{01full}
For any isomorphism-invariant closed subspace $\S\subseteq \G$, the following conditions are equivalent.
\begin{enumerate}
\item[(a)] For any non-empty open sets $U$, $V$ in $\S$, there is a finitely generated group $G$ such that $[G]\cap U\ne \emptyset$ and $[G]\cap V\ne \emptyset$.
\item[(b)] There exists a finitely generated group $G$ such that ${[G]}$ is dense in $\S$.
\item[(c)] The set of marked groups $(G,A)\in \S$ such that $\overline{[G]}=\S$ is comeager in $\S$.
\item[(d)] $\S $ satisfies the zero-one law for $\mathcal L_{\omega_1, \omega}$-sentences.
\item[(e)] For any countable $\L_{\omega_1, \omega}$-theory $F$, $\S$ contains a comeager $F$-equivalence class.
\end{enumerate}
\end{thm}

\begin{proof}
Since $\S$ is closed in $\G$, $\S$ is a Polish space. In particular, it is a Baire space.

We first assume that (a) holds. By our assumptions and Lemma \ref{Aut}, $\S$ is $\Aut $-invariant and the action of $\Aut $ on $\S$ is topologically transitive. Thus (c) follows from Lemma \ref{tt}. Since $\S$ is a Baire space, every comeager subset of $\S$ is non-empty and we obtain (b). Obviously, (b) implies (a). Thus conditions (a), (b), and (c) are equivalent.

Further, we prove that (a) implies (d). Note that for any $\L_{\omega_1, \omega}$-sentence $\sigma$, the set $\ModS(\sigma)$ is $\Aut$-invariant and Borel by Proposition \ref{Borel}. Applying Theorem \ref{01top} to the action of $\Aut $ on $\S$ we obtain that $\ModS(\sigma)$ is either meager or comeager.

Now suppose that (d) holds. We repeat the argument used in the introduction to prove Proposition \ref{Th-gen}. Fix a countable theory $F\subseteq \L_{\omega_1, \omega}$ and define
$$
\overline{F}=F\cup \{ \lnot \sigma \mid \sigma \in F\}
$$
and
$$
Th^{gen}_{\overline{F}}(\S)=\{ \sigma \in \overline{F}\mid \ModS(\sigma) {\rm \; is \; comeager \; in \; } \S\}.
$$
Since $\overline{F}$ is countable and any countable intersection of comeager sets is again comeager, $Th^{gen}_{\overline{F}}(\S)$ has a comeager set of models in $\S$. We will show that any two models of $Th^{gen}_{\overline{F}}(\S)$ are $F$-equivalent, thus proving (e). Arguing by contradiction, assume that there are models $G$ and $H$ of $Th^{gen}_{\overline{F}}(\S)$ and $\sigma \in F$ such that $G\models \sigma $ while $H\not\models \sigma$. We obviously have $${\rm Mod}_\S(\sigma)\cup {\rm Mod}_\S(\lnot\sigma)=\S.$$ Since $\S$ is a Baire space,  it cannot be covered by a union of two meager sets. Together with the zero-one law, this implies that either $\sigma\in Th^{gen}_{\overline{F}}(\S)$ or $\lnot\sigma \in Th^{gen}_{\overline{F}}(\S)$. Both cases contradict the assumption that $G$ and $H$ are models of $Th^{gen}_{\overline{F}}(\S)$.

Finally, we show that (e) implies (a). Let $U$, $V$ be any non-empty open sets in $\S$. Then $U=U_0\cap \S$ for some open $U_0\subseteq \G$. The set $$W=\bigcup_{a\in \Aut} aU_0$$ is open in $\G$ and $\Aut$-invariant. By Proposition \ref{Borel}, $W=\ModG (\sigma)$ for some $\L_{\omega_1, \omega}$-sentence $\sigma$. By condition (e) applied to $F=\{ \sigma\}$, the set $\ModS(\sigma) =W\cap \S$ is either meager or comeager in $\S$. It cannot be meager as $U\subseteq \ModS(\sigma)$ (see property (B$_2$) after the definition of a Baire space).  Hence, $\ModS(\sigma)$ is comeager in $\S$. In particular, it is dense in $\S$ and we have $\ModS(\sigma)\cap V\ne \emptyset$. It follows that there exists $a\in \Aut$ such that $aU\cap V\ne \emptyset$, which is equivalent to (a).
\end{proof}

We mention one particular corollary of the zero-one law, which makes use of the full strength of the $\L_{\omega_1,\omega}$-logic (for an applications, see Proposition \ref{Nek}).

\begin{cor}\label{profin}
Suppose that $\S$ is a closed isomorphism-invariant subspace of $\G$ satisfying the zero-one law for $\L_{\omega_1, \omega}$-sentences.  Then $\S$ contains a comeager subset of groups with isomorphic profinite completions.
\end{cor}
\begin{proof}
It is easy to see that for every finite group $K$, there exists an $\L_{\omega_1, \omega}$-sentence $\phi_K$ such that $G\models \phi_K$ for a finitely generated group $G$ if and only if $K$ is a quotient of $G$\footnote{The first order logic is not sufficient here. Indeed, the free groups of rank $2$ and $3$ are elementarily equivalent \cite{KM,Sel06} but have distinct sets of finite quotients.}. Let $F=\{ \phi_K\}$, where $K$ ranges over the set of all finite groups. By condition (e) from Theorem \ref{01full}, we conclude that generic groups from $\S$ have the same set of finite images. For finitely generated groups, this property is equivalent to having isomorphic profinite completions by the classical theorem of Dixon, Formanek, Poland, and Ribes \cite{DFPR}.
\end{proof}

\paragraph{5.3. Generic properties.} The question of which sentences hold generically in $\S$ in the settings of Theorem \ref{01} is rather non-trivial. We begin by an example showing that the ``naive" attempt to answer this question does not work.

\begin{ex}\label{Sc}
Let $G$ be a finitely generated group. One might expect that an $\L_{\omega_1,\omega}$-sentence holds generically in $\overline{[G]}$ if and only if it holds for $G$. In general, this is false. Indeed, let $G$ be a condensed group. By Theorem \ref{ScTh}, there exists $\sigma \in \L_{\omega_1, \omega}$ such that $\ModS(\sigma)=[G]$. However, $[G]$ is countable and, therefore, cannot be comeager in $\S$.
\end{ex}

On the other hand, we have the following elementary corollary of Proposition \ref{Borel}.

\begin{cor}\label{Pi2}
Let $G$ be a finitely generated group.
\begin{enumerate}
\item[(a)] If $G$ satisfies an $\L_{\omega_1,\omega}$-sentence $\sigma$ of complexity  $\Pi_1$, then $\sigma$ holds for every group from $\overline{[G]}$.
\item[(b)] If $G$ satisfies an $\L_{\omega_1,\omega}$-sentence $\sigma$ of complexity $\Pi_2$, then $\sigma $ holds generically in $\overline{[G]}$.
\end{enumerate}
\end{cor}
\begin{proof}
Part (a) for $\Pi _1$-sentences follows immediately from the fact that for any such a sentence $\sigma$, the set $\ModG(\sigma)$ is closed. To prove (b) we note that  every $\Pi_2$-definable subset of $\G$ is a $G_\delta$-set by Proposition \ref{Borel}. Since $G\models \sigma$, the set ${\rm Mod}_{\overline{[G]}} (\sigma)$ is comeager in $\overline{[G]}$.
\end{proof}

Note that part (b) cannot be extended beyond $\Pi_2$ even if we restrict to the first-order logic. Indeed, methods of \cite{MO} can be used to show that there exists an acylindrically hyperbolic group $G$ such that  $\overline{[G]}=\overline{\mathcal{AH}}_{tf}$. However, a generic group in $\overline{\mathcal{AH}}_{tf}$ has $2$ conjugacy classes by \cite{Hull}. The latter property can be expressed by the $\forall\exists$-sentence
$$
\zeta = \forall\, a\; \forall\, b\; \exists \, t\; (a=1 \vee b=1 \vee t^{-1}at=b)
$$
and $\zeta$  is never satisfied by an acylindrically hyperbolic group (in fact, every acylindrically hyperbolic group has exponentially growing set of conjugacy classes \cite{HO}). It follows that the sentence $\lnot \zeta$ is satisfied by $G$ but does not belong to $Th^{gen}(\overline{[G]})$. Note that $\lnot \xi$ is equivalent to a $\exists\forall$-sentence. We leave details to the curious reader.

\section{Examples and applications}

\paragraph{6.1. Condensed groups.}
Our next goal is to discuss non-trivial instances of the zero-one law for $\L_{\omega_1,\omega}$-sentences. As explained in the introduction, this leads to the notion of a condensed group via the following immediate corollary of Proposition \ref{Aut}.

\begin{cor}\label{dych}
For every finitely generated group $G$, the isomorphism class $[G]$ is either discrete or has no isolated points.
\end{cor}

\begin{proof}
If $[G]$ is not discrete, there exists an accumulation point $(G,A)\in [G]$. Then every point of $[G]$ is an accumulation point since $[G]$ is the $Aut_f(F_\infty)$-orbit of $(G,A)$ and the action of $Aut_f(F_\infty)$ on $\G$ is continuous.
\end{proof}

We begin by proving Proposition \ref{no-cond}. In fact, we obtain a somewhat stronger result. Before stating it, we have to recall several definitions. A group $G$ is \emph{equationally Noetherian}, if every system of equations with parameters in $G$ has the same set of solutions as a finite subsystem \cite{BMR}. Examples of equationally Noetherian groups include linear groups over commutative, Noetherian, unital rings (e.g., fields) \cite{BMR}, finitely generated abelian-by-nilpotent groups \cite{Bry}, hyperbolic groups \cite{RW} and some of their generalizations (see \cite{GH} and references therein).

A group $G$ is \emph{residually finite} if for every non-trivial element $g\in G$, there is a homomorphism $\e\colon G\to K$ to a finite group $K$ such that $\e(G)\ne 1$. Every finitely generated linear group is residually finite by the classical result of Maltsev. A group $G$ is \emph{Hopfian} if every epimorphism $G\to G$ is injective. Finitely generated residually finite (or equationally Noetherian) groups also serve as main examples of Hopfian groups. A simple example of a group that fails to be equationally Noetherian, residually finite, and Hopfian is the Baumslag-Solitar group
$$
G=\langle a, b\mid b^{-1}a^2b=a^3\rangle .
$$
For more on this, we refer to \cite[Chapter IV, Section 4]{LS}.

We say that a group $G$ is \emph{extremely non-Hopfian} if for every finite subset $\mathcal F\subseteq G$, there exists an non-injective epimorphism $G\to G$ whose restriction to $\mathcal F$ is injective.

\begin{prop}\label{cond}
\begin{enumerate}
\item[(a)] A finitely generated equationally Noetherian group cannot be condensed. In particular, finitely generated linear and abelian-by-nilpotent groups are not condensed.
\item[(b)] A finitely presented condensed group is extremely non-Hopfian. In particular, a finitely presented residually finite group cannot be condensed.
\end{enumerate}
\end{prop}

\begin{proof}
Part (a) follows from the fact that for every equationally Noetherian group $G$, there are only countably many $G$-limit groups, which implies that $\overline{[G]}$ is countable (see \cite{OH07} for details). By the Baire category theorem, this is impossible if $G$ is condensed. Part (b) follows easily from Proposition \ref{Quot} and the definition of a condensed group.
\end{proof}

Next, we discuss examples of condensed groups.

\begin{prop}\label{enH}
Every finitely generated extremely non-Hopfian group is condensed.
\end{prop}
\begin{proof}
Let $(G,A)\in \G$, where $G$ is a finitely generated extremely non-Hopfian group and let $U$ be any neighborhood of $(G,A)$ in $\G$. By the definition of the topology on $\G$, there exists $r\in \mathbb N$ such that every $(H,B)\in \G$ satisfying $(H,B)\approx _r(G,A)$ belongs to $U$. Let $\mathcal F $ be the set of all elements of $G$ length at most $r$ with respect to $A$. By our assumption, there is a non-injective epimorphism $\e\colon G\to G$ such that the restriction of $\e$ to $\mathcal F$ is injective. It follows that $(G, \e(A))\approx _r(G,A)$. Note that $(G,\e(A))\ne (G,A)$ as $\e$ is not injective. Thus $(G,A)$ is a limit point of $[G]$, i.e., $G$ is condensed.
\end{proof}

We now derive part (a) of Example \ref{excond}.

\begin{cor}
Let $G$ be a finitely generated group such that $G\cong G\times G$. Then $G$ is condensed.
\end{cor}

\begin{proof}
Let $\mathcal F$ be a non-empty finite subset of $G$ and let $k=|\mathcal{FF}^{-1}|$. By induction, we have $G=G_1 \times \cdots \times G_{k}$, where $G_i\cong G$ for all $i$. Since $G_i\cap G_j=\{1\}$ for $i\ne j$, we have $\mathcal{FF}^{-1}\cap G_i=\{ 1\} $ for some $i$. Hence, the natural homomorphism $G\to G/G_i$ is injective on $\mathcal F$. Clearly, $G/G_i\cong G$. Thus $G$ is extremely non-Hopfian. It remains to apply Proposition \ref{enH}.
\end{proof}

The second part of Example \ref{excond} is based on the notion of an iterated monodromy group introduced by Nekrashevich. We do not go into detail here and refer the interested reader to the survey \cite{Nek11} for definitions. In \cite{Nek}, Nekrashevich constructed a closed subspace $\mathcal N\subseteq \G$ homeomorphic to the Cantor set such that the isomorphism class of the iterated monodromy group $IMG(z^2+i)$ of the polynomial $z^2+i$ is dense in $\mathcal N$. In particular, $IMG(z^2+i)$ is condensed. Later Nekrashevich used this construction to prove the following.

\begin{prop}[Nekrashevich, \cite{Nek14}]\label{Nek}
There exist $2^{\aleph_0}$ pairwise non-isomorphic residually finite groups of intermediate growth having isomorphic profinite completions.
\end{prop}

The proof given in \cite{Nek14} essentially uses the algebraic structure of groups from $\mathcal N$. The key step is an explicit description of all finite quotients of such groups, which can be obtained from their actions on rooted trees. To illustrate the power of the $\L_{\omega_1, \omega}$-logic, we sketch an alternative proof of Proposition \ref{Nek}, which only uses results of \cite{Nek} and the present paper.

\begin{proof}
It is shown in \cite{Nek} that all groups in $\mathcal N =\overline{[IMG(z^2+i)]}$ are residually finite and the group $IMG(z^2+i)$ has intermediate growth. We first prove that the property of having intermediate growth is generic in $\mathcal N$.

Let $\mathcal{SE}$ denote the set of all $(G,A)\in \G$ such that $G$ has subexponential growth. Since the growth function of every group is submultiplicative, we have $(G,A)\in \mathcal{SE}$ if and only if for any rational $a>1$, there exists $N\in \mathbb N$ such that for all $n>N$, the number of elements of length at most $n$ in $G$ with respect to $A$ is less than $a^n$. It is not difficult to see that the latter condition can be expressed by an $\L_{\omega_1,\omega}^\prime$-sentence of complexity class $\Pi_2$ (see the proof of Proposition \ref{Borel} for the definition of the extended language $\L^\prime$). By Lemma \ref{claim}, $\mathcal{SE}$ is a $G_\delta $-set. Since $IMG(z^2+i)$ has subexponential growth and its isomorphism class is dense in $\mathcal N$, the set $\mathcal{SE}$ is comeager in $\mathcal N$.

By the Gromov theorem, groups of polynomial growth are virtually nilpotent and, in particular, finitely presented. Hence, they cannot occur in $\mathcal N$ by Proposition \ref{Quot}. Thus generic groups in $\mathcal N$ actually have intermediate growth. By Corollary \ref{profin}, generic groups in $\mathcal N$ also have isomorphic profinite completions and the result follows.
\end{proof}

Finally, we discuss the relation between smoothness of the isomorphism relation and the existence of condensed groups in subspaces of $\G$. Let $\{ 0,1\}^\omega$ denote the set of all binary sequences and let $E_0$ denote the equivalence relation on $\{ 0,1\}^\omega$ such that $xE_0y$ if and only if the sequences $x$ and $y$ match on all but finitely many terms. We will need the following fundamental result from the theory of Borel equivalence relations proved by  Harrington, Kechris, and Louveau, following an earlier work of Glimm and Effros.

\begin{thm}[Harrington--Kechris--Louveau, \cite{HKL}]\label{GE}
Let $E$ be a Borel equivalence relation on a Polish space $X$. Then exactly one of the following holds.
\begin{enumerate}
\item[(a)] $E$ is smooth.
\item[(b)] There is a continuous injective map $f\colon \{0,1\}^\omega \to X$ such that for all $x,y\in \{ 0,1\}^\omega$, we have $xE_0y$ if and only if $f(x)Ef(y)$.
\end{enumerate}
\end{thm}

\begin{proof}[Proof of Proposition \ref{smooth}]
We first prove the ``only if" direction. Arguing by contradiction, suppose that $(G,A)\in \S$ is condensed and the isomorphism relation on $\S$ is smooth. Then the isomorphism relation on $\overline{[G]}$ is smooth. That is, there exists a Polish space $P$ and a Borel function $f\colon \overline{[G]} \to P$ such that $f(G,A)=f(H,B)$ if and only if $G\cong H$. Every Borel map between Polish spaces is ``generically continuous" \cite[Theorem 8.38]{Kech}. This means that there is a comeager subset $\mathcal J$ of $\overline{[G]}$ such that $f\vert_{\mathcal J}$ is continuous. Consider the set
$$
\mathcal I =\bigcap_{a\in \Aut } a\mathcal J\subseteq \S.
$$
Clearly, $\mathcal I$ is isomorphism-invariant. Since $\Aut $ is countable, the set $\mathcal I$ is comeager in $\overline{[G]}$. The action of $\Aut $ on $\overline{[G]}$ is topologically transitive, hence the set of points with dense orbits is also comeager in $\overline{[G]}$ by Lemma \ref{tt}. In particular, there is $(H,B)\in \mathcal I$ such that $[H]$ is a dense subset of  $\mathcal I$. Since $f\vert_{\mathcal I}$ is continuous and constant on isomorphism classes, $f(\mathcal I)$ consists of a single point, i.e., all groups in $\mathcal I$ are isomorphic. However, this contradicts countability of isomorphism classes. Indeed, by the Baire theorem, $\mathcal I$ is uncountable being a comeager subset of the Polish space $\overline{[G]}$ without isolated points.

Assume now that the isomorphism relation on $\S$ is not smooth and let $f\colon \{0,1\}^\omega\to \S$ be the map provided by Theorem \ref{GE} applied to the isomorphism relation on $\S$. Since all equivalence classes are non-discrete in $\{0,1\}^\omega$, the image of every element of $\{0,1\}^\omega$ in $\S$ is a condensed group.
\end{proof}

\paragraph{6.2. Spaces associated to hyperbolic groups and their generalizations.}
To prove our next result, Theorem \ref{hyp}, we will need a family of hyperbolic groups defined as follows. For $p\in \NN$, let
$$
A_p=\langle a_1, \ldots, a_p \mid a_i^p=1,\; [a_i,a_j]=1, \; i,j=1, \ldots, p\rangle\cong (\ZZ/p\ZZ)^p,
$$
and
$$
B_p=\langle c,d\mid c^p=1,\; d^p=1\rangle \cong \ZZ/p\ZZ\ast \ZZ/p\ZZ.
$$
Further, let
$$H_p\cong A_p \rtimes B_p,$$
be the split extension corresponding to the action of $c$ and $d$ on $A_p$ by the cyclic permutation of the generators $a_1, \ldots, a_p$. It is easy to see that the groups $H_p$ are non-elementary hyperbolic for all $p>2$. Therefore, they are non-elementary relatively hyperbolic (with respect to $\{ 1\}$), and acylindrically hyperbolic.

We fix the generating set $X_p=\{ a_1, \ldots, a_p, c,d\}$ of $H_p$ and let
$$
U_p=\{ (H,X)\in \G \mid (H,X)\approx _{\max\{p,2\}} (H_p,X_p)\}.
$$
Let also $S_p$ denote the permutation group on $p$ symbols.

\begin{lem}\label{fip}
For every prime $p\in \NN$, there is a first-order sentence $\phi _p$ in the language of groups such that the following conditions hold.
\begin{enumerate}
\item[(a)] If a group $G$ satisfies $\phi_p$, then there is a homomorphism $G\to S_p$ with non-trivial image.
\item[(b)] For any $(H,X)\in U_p$, the group $H$ is generated by elements of order $p$ and $H\models \phi_p$.
\end{enumerate}
\end{lem}
\begin{proof}
We write $a\ne b$ to abbreviate $\lnot a=b$ and consider the formulas
$$
\psi_p=\left(\bigwedge_{i=1}^p x_i^p=1\right)\wedge \left(\bigwedge_{i,j=1}^p x_ix_j=x_jx_i\right)
$$
and
$$
\phi_p=\exists\, x_1\; \ldots\;\exists\, x_p\; \psi_p \wedge \left(\forall\, g\; \left(\bigwedge_{i=1}^p \bigvee_{j=1}^{p} x_ig =gx_j\right) \wedge \exists\, h\; x_1h\ne hx_1\right).
$$
It is easy to see that $G\models \phi_p$ if and only if there exist elements $x_1, \ldots, x_p\in G$ such that $K=\langle x_1, \ldots, x_p\rangle $ is a normal subgroup of $G$ isomorphic to a quotient group of $(\ZZ/p\ZZ)^p$, conjugation by any element of $G$ on $K$ permutes the generators, and at least one element $h$ acts on $\{ x_1, \ldots, x_p\}$ non-trivially. In particular, we have (a).

Further, suppose that $$(H,X)\approx _{\max\{p,2\}} (H_p,X_p),$$ where $X=\{a_1^\prime, \ldots, a_p^\prime, c^\prime, d^\prime\}$. By the definition of the relation $\approx _{\max\{p,2\}}$, the elements $x_1=a_1^\prime, \ldots, x_p=a_p^\prime$ satisfy $\psi_p$, and conjugation by every $x\in X$ permutes these elements. It follows that conjugation by every element of $H$ permutes them. Using the definition of the relation $\approx _{\max\{p,2\}}$ again, we have $a_1^\prime c^\prime\ne c^\prime a_1^\prime$. Thus $H\models \phi_p$.
\end{proof}

\begin{proof}[Proof of Theorem \ref{hyp}]
The spaces shown on diagram (\ref{diag}) have no isolated points by Lemma \ref{dense}.

Suppose that one of the sets $\overline{\mathcal{H}}$, $\overline{\mathcal{RH}}$, $\overline{\mathcal{AH}}$ contains a comeager subset covered by finitely many elementary equivalence classes $E_1, \ldots , E_k$. Then for every $p>2$, the marked group $(H_p, X_p)$ belongs to $\overline{E}_1\cup \ldots \cup \overline{E}_k$. In particular, there exist primes $p>q$ and $i\in \{ 1, \ldots, k\}$ such that $(H_p,X_p), (H_q, X_q)\in \overline{E}_i$.

\begin{figure}
 \centering
\begingroup%
  \makeatletter%
  \providecommand\color[2][]{%
    \errmessage{(Inkscape) Color is used for the text in Inkscape, but the package 'color.sty' is not loaded}%
    \renewcommand\color[2][]{}%
  }%
  \providecommand\transparent[1]{%
    \errmessage{(Inkscape) Transparency is used (non-zero) for the text in Inkscape, but the package 'transparent.sty' is not loaded}%
    \renewcommand\transparent[1]{}%
  }%
  \providecommand\rotatebox[2]{#2}%
  \newcommand*\fsize{\dimexpr\f@size pt\relax}%
  \newcommand*\lineheight[1]{\fontsize{\fsize}{#1\fsize}\selectfont}%
  \ifx\svgwidth\undefined%
    \setlength{\unitlength}{322.87187858bp}%
    \ifx\svgscale\undefined%
      \relax%
    \else%
      \setlength{\unitlength}{\unitlength * \real{\svgscale}}%
    \fi%
  \else%
    \setlength{\unitlength}{\svgwidth}%
  \fi%
  \global\let\svgwidth\undefined%
  \global\let\svgscale\undefined%
  \makeatother%
  \begin{picture}(1,0.25861219)%
    \lineheight{1}%
    \setlength\tabcolsep{0pt}%
    \put(0,0){\includegraphics[width=\unitlength,page=1]{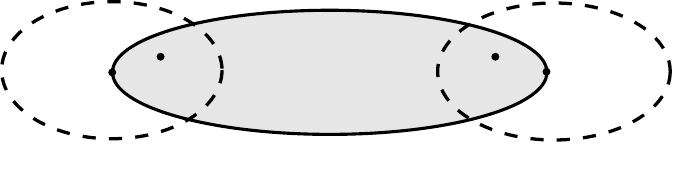}}%
    \put(0.02455712,0.15457323){\color[rgb]{0,0,0}\makebox(0,0)[lt]{\lineheight{1.25}\smash{\begin{tabular}[t]{l}$(H_p, X_p)$\end{tabular}}}}%
    \put(0.21376157,0.132105){\color[rgb]{0,0,0}\makebox(0,0)[lt]{\lineheight{1.25}\smash{\begin{tabular}[t]{l}$(H,X)$\end{tabular}}}}%
    \put(0.45955754,0.14485952){\color[rgb]{0,0,0}\makebox(0,0)[lt]{\lineheight{1.25}\smash{\begin{tabular}[t]{l}$\overline{E}_i$\end{tabular}}}}%
    \put(0.66294756,0.1331186){\color[rgb]{0,0,0}\makebox(0,0)[lt]{\lineheight{1.25}\smash{\begin{tabular}[t]{l}$(K,Y)$\end{tabular}}}}%
    \put(0.83187992,0.15111036){\color[rgb]{0,0,0}\makebox(0,0)[lt]{\lineheight{1.25}\smash{\begin{tabular}[t]{l}$(H_q, X_q)$\end{tabular}}}}%
    \put(0.15750631,0.00903895){\color[rgb]{0,0,0}\makebox(0,0)[lt]{\lineheight{1.25}\smash{\begin{tabular}[t]{l}$U_p$\end{tabular}}}}%
    \put(0.81363534,0.00498456){\color[rgb]{0,0,0}\makebox(0,0)[lt]{\lineheight{1.25}\smash{\begin{tabular}[t]{l}$U_q$\end{tabular}}}}%
  \end{picture}%
\endgroup%

    \caption{}\label{fig2}
\end{figure}

It follows that there are $(H,X)\in U_p$ and $(K,Y)\in U_q$ such that $H\equiv K$ (see Fig. \ref{fig2}). By Lemma \ref{fip} (b), we have $K\models \phi_q$. Hence $H\models \phi_q$. In particular, $H$ admits a homomorphism to $S_q$ with non-trivial image by part (a) of Lemma \ref{fip}. However, this is impossible since $H$ is generated by elements of prime order $p>q$ by part (b). This contradiction shows that the number of elementary equivalence classes in every comeager subset of each of the classes $\overline{\mathcal{H}}$, $\overline{\mathcal{RH}}$, or $\overline{\mathcal{AH}}$ is infinite.

Finally, let $\mathcal Z$ be one of the sets $ {\mathcal{H}_0}$, $ {\mathcal{H}_{tf}}$, $ {\mathcal{RH}_0}$, $ {\mathcal{RH}_{tf}}$, $ {\mathcal{AH}_0}$, or $ {\mathcal{AH}_{tf}}$. Given any non-empty open subsets $U,V\subseteq \overline{\mathcal Z}$, we can find marked groups $(G_1, A_1), (G_2, A_2)\in \mathcal Z$ such that $(G_1,A_1)\in U$ and $(G_2,A_2)\in V$. Let $r$ be a natural number such that
\begin{equation}\label{UV}
U_{G_1,A_1}(r)\subseteq U\;\;\; {\rm and}\;\;\; U_{G_2,A_2}(r)\subseteq V,
\end{equation}
where the sets  $U_{G_i,A_i}(r)$ are defined by (\ref{UGA}).
Let $${\mathcal F_i}=\{ g\in G_i\mid |g|_{A_i}\le r\}, \;\;\; i=1,2.$$ Let $Q$ and $\e_i\colon G_i\to Q$ be the group and the epimorphisms provided by Lemma \ref{Q}. Then $(Q, \e_i(A_i))\approx_r(G_i, A_i)$. By (\ref{UV}), we obtain $(Q, \e_1(A_1))\in U$ and  $(Q, \e_2(A_2))\in V$. This verifies condition (a) from Theorem \ref{01} and part (b) of Theorem \ref{hyp} follows.
\end{proof}

\paragraph{6.3. Lacunary hyperbolic groups.}
A finitely generated group is \emph{lacunary hyperbolic} if one of its asymptotic cones is an $\mathbb R$-tree. Recall that an asymptotic cone of a finitely generated group is a metric space, depending on the choice of a non-principal ultrafilter, which shows how the group looks like from ``infinitely far away". We do not go into detail here and refer the interested reader to \cite{OOS} instead.  Readers unfamiliar with asymptotic cones can accept the equivalent characterization given in Theorem \ref{LH-def} below as the definition of lacunary hyperbolicity.

We say that a group $(K,Z)\in \G$ is the \emph{limit of an epimorphic sequence}
\begin{equation}\label{lh}
(K_1, Z_1) \stackrel{\e_1}\longrightarrow (K_2,Z_2)\stackrel{\e_2}\longrightarrow \ldots,
\end{equation}
where $(K_i,Z_i)\in \G$ for all $i$, if every $\e_i\colon K_i\to K_{i+1}$ is an epimorphism, $\e_i(Z_i)=Z_{i+1}$, $$K\cong K_1/\bigcup_{i\in \NN} Ker(\e_1 \circ \cdots \circ \e_i),$$ and the natural homomorphism $K_1\to K$ maps $Z_1$ to $Z$.

\begin{thm}[{\cite[Theorem 1.1]{OOS}}]\label{LH-def}
For every group $K$, the following conditions are equivalent.
\begin{enumerate}
\item[(a)] $K$ is lacunary hyperbolic.
\item[(b)] There exists a finite generating set $Z$ of $K$ such that $(K,Z)$ is the limit of an epimorphic sequence (\ref{lh}) and  there exist positive constants $\delta_i$, $r_i$ such that for all $i\in \mathbb N$, we have:
\begin{enumerate}
\item[($i$)] the Cayley graph $\Gamma (K_i,Z_i)$ is $\delta_i$-hyperbolic;
\item[($ii$)] the epimorphism $\e_i$ is injective on the subset $\{k\in K_i \mid |k|_{Z_i}\le r_i\}$;
\item[($iii$)] $\lim_{i\to \infty} r_i/\delta_i= \infty$.
\end{enumerate}
\item[(c)] The same property as in (b) holds for every finite generating set $Z$ of $K$.
\end{enumerate}
\end{thm}

\begin{proof}[On the proof]
Formally speaking, \cite[Theorem 1.1]{OOS} only claims that (a) is equivalent to (b). However, the proof of Theorem 1.1 in \cite{OOS} goes as follows. Starting with any generating set $Z$ of $K$ one writes an (infinite) presentation of $K$ with respect to $Z$ and then obtains the desired sequence of groups $K_i$ by truncating this presentation. Thus the proof actually shows that (a) implies (c). Since (c) obviously implies (b), we get the equivalence of all three conditions.
\end{proof}

\begin{cor}\label{hlh}
The spaces $\mathcal {LH}$, $\mathcal {LH}_{0}$, and $\mathcal {LH}_{tf}$ are subsets of $\overline{\mathcal H}$, $\overline{\mathcal H}_{0}$, and $\overline{\mathcal H}_{tf}$, respectively.
\end{cor}

\begin{proof}
For $\mathcal {LH}$, this is an immediate corollary of Theorem \ref{LH-def}.

Further, suppose that $(K,Z)\in \mathcal {LH}_{0}$ or $(K,Z)\in \mathcal {LH}_{tf}$ and let (\ref{lh}) be the sequence provided by Theorem \ref{LH-def}. Let $k\in K_i$ be an element of finite order for some $i\in \NN$. Then $k$ is conjugate to an element $h\in K_i$ of length $|h|_{Z_i} \le 4\delta_i+2$ (see the last line of the proof of \cite[Theorem 3.2, Chapter III.$\Gamma$]{BH}). Therefore, condition ($iii$) implies that for all but finitely many indices $i$, every non-trivial finite order element of $K_i$ is mapped to a non-trivial finite order element of $K$ under the natural homomorphism $K_i\to K$. Passing to a subsequence, we can assume that this property holds for all $i$.

In particular, if $K$ is torsion-free, then all groups $K_i$ are torsion-free. This settles the case of $\mathcal {LH}_{tf}$. Furthermore, if $\ll k\rr$ is finite and non-trivial in $K_i$, then the natural image of $k$ in $K$ generates a non-trivial finite normal subgroup. Thus we have $(K_i,Z_i)\in \mathcal {LH}_0$ whenever $(K,Z)\in \mathcal {LH}_0$.
\end{proof}

Not every limit of hyperbolic groups in $\G$ is lacunary hyperbolic. We discuss one particular example, which will be used later.

\begin{ex}\label{W}
Fix some $k\in \NN$ and consider the group
$$
W(k)=\langle u, v \mid R_1, R_2, \ldots\rangle ,
$$
where $R_j$ are defined by (\ref{wj}) for all $j\in\NN$. Similarly to the presentations of groups $W_n(k)$ considered in Example \ref{Wnk}, this infinite presentation satisfies the $C^\prime(1/6)$ small cancellation condition for $k\ge 30$. By \cite[Proposition 3.12]{OOS}, $W(k)$ is not lacunary hyperbolic. The idea behind this fact is that $W(k)$ has relations ``at all scales". This implies the existence of non-trivial simple loops in all asymptotic cones and, therefore, none of the asymptotic cones of $W(k)$ is an $\mathbb R$-tree. On the other hand, $(W(k), \{ u,v\})$ is the limit of the sequence of hyperbolic groups $(W_n(k), \{ u,v\})$ considered in Example \ref{Wnk}. Thus $W(k)$ belongs to $\overline{\mathcal{H}}$. Moreover, it is well-known that a group given by a presentation satisfying $C^\prime(1/6)$ is torsion-free provided none of the relations is a proper power (see, for example, \cite{Hue}). It follows that $W_n(k)\in \mathcal H_0$ and $W(k)\in \overline{\mathcal{H}}_0$.
\end{ex}

The next lemma will only be used to show that $\mathcal{LH}$, $\mathcal{LH}_0$, and $\mathcal{LH}_{tf}$ are homeomorphic to the subspace of irrational numbers in $\mathbb R$. Readers not interested in this fact can skip the lemma and other results below and go directly to the proof of Theorem \ref{LHgen}.

\begin{lem}\label{HLH}
The sets $\overline{\mathcal H}\setminus \mathcal{LH}$,  $\overline{\mathcal H}_0\setminus \mathcal{LH}_0$, and $\overline{\mathcal H}_{tf}\setminus \mathcal{LH}_{tf}$ are dense in $\overline{\mathcal H}$, $\overline{\mathcal H}_{0}$, and $\overline{\mathcal H}_{tf}$, respectively.
\end{lem}

\begin{proof}
Recall that we write $G\in \S$ for a group $G$ and a subset $\S\subseteq \G$ if $(G,A)\in \S$ for some finite generating set $A$.

We deal with the case of torsion-free groups first. It suffices to show that for every $G \in \mathcal H_{tf}$ and every finite $\mathcal K\subseteq G$, there is $Q\in \overline{\mathcal H}_{tf}\setminus \mathcal{LH}_{tf}$ and an epimorphism $G\to Q$ injective on $\mathcal K$. Let $(G,X)\in \mathcal H_{tf}$. Since $G$ is non-elementary hyperbolic, the generating set $X$ satisfies condition (AH$_2$) from Theorem \ref{aa} (this uses the obvious observation that the action of $G$ on $\Gamma(G,X)$ is acylindrical whenever $|X|<\infty $). By Proposition \ref{KG}, there exists a subgroup $H\cong F_2=\langle x,y\rangle $ such that $H\h (G,X)$; note that we have $K(G)=\{1\}$ as $G$ is torsion-free. Let $\mathcal F=\mathcal F(\mathcal K)\subseteq G\setminus\{ 1\}$ be the finite set provided by Theorem \ref{CEP} applied to the group $G$ and the subgroup $H$. Let $W(k)$ be the group given in Example \ref{W}. By the Greedlinger lemma (see Lemma \ref{Green}), there exists $k\ge 30$ such that the kernel of the map $F_2\to W(k)$ sending $x\mapsto u$ and $y\mapsto v$ does not intersect $\mathcal F$. From now on, we fix any $k$ satisfying this condition.

Let $N$ (respectively, $N_n$) denote the kernel of the map $F_2\to W(k)$ (respectively, $F_2\to W_n(k)$) sending $x\mapsto u$ and $y\mapsto v$. We have $N_n\subseteq N$ and hence $N_n\cap \mathcal F=\emptyset$. Thus the conclusion of Theorem \ref{CEP} holds for all normal subgroups $N_n$ as well as for $N$. We denote by $Y$ (respectively, $Y_n$) the natural image of the set $X$ in $Q=G/\ll N\rr$ (respectively, $Q_n=G/\ll N_n\rr$). Since $N=\bigcup N_n$, we have $$\lim_{n\to \infty}(Q_n, Y_n)= (Q, Y)$$ in $G$ (see Example \ref{exconv} (b)). As we already mentioned in Example \ref{W}, $W(k)$ is torsion-free and every $W_n(k)$ is torsion-free hyperbolic. Since all generating sets under consideration are finite, parts (a) and (b) of Theorem \ref{CEP} imply that $Q$ (respectively, $Q_n$) is hyperbolic relative to $W(k)$ (respectively, $W_n(k)$); in addition, these groups are torsion-free by part (d). By Proposition \ref{rhh}, every $Q_n$ is hyperbolic. It follows that $(Q,Y)\in  \overline{\mathcal H}_{tf}$.

Since $Q$ is hyperbolic relative to $W(k)$, $W(k)$ is undistorted in $Q$ with respect to a finite generating set of $Q$ by Proposition \ref{undist}. By \cite[Theorem 3.8]{OOS}, every finitely generated subgroup of a lacunary hyperbolic group undistorted with respect to a finite generating set of the group is lacunary hyperbolic itself. Since $W(k)$ is not lacunary hyperbolic, we conclude that $Q$ is not lacunary hyperbolic. This finishes the proof of the lemma for $\overline{\mathcal H}_{tf}\setminus \mathcal{LH}_{tf}$.

For $\overline{\mathcal H}_0\setminus \mathcal{LH}_0$, we start with any $(G,X)\in \mathcal H_{0}$ and argue as above. We still have $K(G)=\{ 1\}$ by the definition of $\mathcal H_0$ in this case. Corollary \ref{K=1} implies that  $Q_n\in \mathcal H_{0}$ and hence $Q\in \overline{\mathcal H}_{0}\setminus \mathcal{LH}_{0}$.

Finally, we note that for $\overline{\mathcal H}\setminus \mathcal{LH}$, exactly the same proof works except that we may have $K(G)\ne \{ 1\}$ and $H\cong F_2\times K(G)\not\cong F_2$ in this case. Instead of maps $F_2\to W(k)$ and $F_2\to W_n(k)$, we consider maps $H\to W(k)\times K(G)$ and $H\to W_n(k)\times K(G)$, which send $x$ and $y$ to $u$ and $v$, respectively, and acts identically on $K(G)$. It is well-known (and straightforward to see from the definitions in terms of asymptotic cones) that the property of being lacunary hyperbolic is stable under taking finite extensions and subgroups of finite index. Thus $W(k)\times K(G)$ is not lacunary hyperbolic. On the other hand,  $W_n(k)\times K(G)$ is hyperbolic for every $n$ (see Example \ref{exhyp} (d)), and we can complete the proof as above.
\end{proof}

Recall that the \emph{Baire set}\footnote{This set is often called ``the Baire space"; we prefer to use the term ``Baire set" (in analogy with the Cantor set) so that it is not confused with the notion of a Baire space discussed in Section 3.1.}  is  the set $\NN^\NN$ equipped with the product metric. It is well-known that $\NN^\NN$ is homeomorphic to the set of irrational numbers with the topology induced from $\mathbb R$. Moreover, we have the following.

\begin{thm}[Alexandrov-Uryson, {\cite[Theorem 7.7]{Kech}}]
Every non-empty Polish zero-dimensional space, where all compact subsets have empty interiors, is homeomorphic to the Baire set.
\end{thm}

We will need an obvious corollary of this theorem.

\begin{cor}\label{CB}
Let $\mathcal P$ be a non-empty zero-dimensional Polish space and let $\mathcal B$ be a $G_\delta$-subset of $\mathcal P$. Suppose that both $\mathcal B$ and $\mathcal P\setminus \mathcal B$ are dense in $\mathcal P$. Then $\mathcal B$ is homeomorphic to the Baire set.
\end{cor}

\begin{proof}
The set $\mathcal B$ is Polish by Proposition \ref{ps-cl} and zero-dimensional being a subspace of a zero-dimensional space. According to the Alexandrov-Uryson theorem, it suffices to show that every compact subset of $\mathcal B$ has empty interior. Arguing by contradiction, suppose that a compact subset $K\subseteq \mathcal B$ contains a non-empty subset $U$ that is open in $\mathcal B$. Let $W$ be an open subset of $\mathcal P$ such that $U=\mathcal B \cap W$. Since $\mathcal P\setminus \mathcal B$ is dense in $\mathcal P$, there is $x\in W\setminus \mathcal B$. Since $\mathcal B$ is dense in $\mathcal P$, there is a sequence of points $x_1, x_2, \ldots \in \mathcal B \cap W\subseteq K$ converging to $x$. Since $K$ is compact, we must have $x\in K \subseteq \mathcal B$, a contradiction.
\end{proof}

We are now ready to prove our last result.

\begin{proof}[Proof of Theorem \ref{LHgen}]
We first show that the spaces $\mathcal{LH}$, $\mathcal{LH}_0$, and $\mathcal{LH}_{tf}$ are dense $G_\delta$-subsets of $\overline{\mathcal H}$, $\overline{\mathcal H}_{0}$, and $\overline{\mathcal H}_{tf}$, respectively. To this end, for each $(G,X)\in \mathcal H$, we denote by $\delta_{G,X}$ the smallest natural number such that the Cayley graph $\Gamma (G,X)$ is $\delta_{G,X}$-hyperbolic. For every $j\in \mathbb N$ and every $(G,X)\in \mathcal H$, we let
$$
U_j(G,X)=\{(H,Y)\in \overline{\mathcal H}\mid (H,Y)\approx_{j(4 \delta_{G,X}+6)}(G,X)\} .
$$
Let
$$
\mathcal{U}=\bigcap_{j\in \mathbb N}\;\bigcup_{(G,X)\in \mathcal H} U_j(G,X).
$$

We  claim that every $(K,Z)\in \mathcal{U}$ is lacunary hyperbolic. Indeed, for each $j\in \mathbb N$, we can find a marked group $(K^\prime _j, Z^\prime_j)\in \mathcal H$ such that $(K,Z)\in U_j(K^\prime _j,Z^\prime _j)$. We define a sequence $(j(i))_{i \in \NN}$ inductively by letting
$$
j(1)=1,\;\;\; j(i+1)=j(i)(4\delta_{K'_{j(i)},Z'_{j(i)}}+6),
$$
and

$$(K_i, Z_i)=(K^\prime_{j(i)}, Z^\prime_{j(i)})\in \mathcal H$$ for $i\ge 1$.
Let also
$$r_i=j(i+1).$$

By our construction, $(K,Z)\approx_{r_i}(K_i,Z_i)$ for all $i$. Note that
$$
r_{i+1} =j(i+2)= j(i+1)(4\delta_{K_{i+1},Z_{i+1}}+6)> j(i+1)=r_i
$$
for all $i \in \NN$. Thus we also have $(K,Z)\approx_{r_i}(K_{i+1}, Z_{i+1})$. It follows that $(K_i, Z_i)\approx_{r_i}(K_{i+1}, Z_{i+1})$.

Being a hyperbolic group, $K_i$ admits a finite presentation $K_i=\langle Z_i\mid \mathcal R_i\rangle$, where every relator $R\in \mathcal R_i$ has length at most $4\delta_{K_i,Z_i}+6$ (see, for example, the proof of \cite[Chapter III.$\Gamma$.2.2]{BH}). Our definition of $r_i$ implies that $r_i \ge 4\delta_{K_{i}, Z_{i}}+6$. Hence,  the map $Z_i\to Z_{i+1}$ extends to an epimorphism $\e_i\colon K_i\to K_{i+1}$ which is injective on the set of all elements $k\in K_i$ of length $|k_i|_{Z_i}\le r_i$. Similarly, there exists an epimorphism $K_i\to K$ that maps $Z_i$ to $Z$ and is injective of the set of elements of length at most $r_i$. Clearly, we have $\lim_{i\to \infty}r_i/\delta_{K_{i}, Z_{i}}= \infty $. In particular, $\lim_{i\to \infty} r_i=\infty$. This easily implies that $K$ is the direct limit of the sequence (\ref{lh}) satisfying conditions (i)--(iii). Thus $\mathcal{U}$ consists of lacunary hyperbolic groups.

Conversely, let $(K,Z)\in \mathcal{LH}$. Let (\ref{lh}) be an epimorphic sequence provided by part (c) of Theorem \ref{LH-def}.  Since all maps $\e_i$ in (\ref{lh}) are onto, $K_i$ is non-elementary for all $i$. Further, conditions ($i$)--($iii$) imply that for every $j\in \NN$, there is $n=n(j)$ such that $(K,Z)\in U_j(K_n,Z_n)$. It follows that $(K,Z)\in \mathcal U$. Thus we have $\mathcal U=\mathcal{LH}$.

By definition, $\mathcal{U}$ is a $G_\delta$-set. Note that $\mathcal U$ is also dense in $\overline{\mathcal H}$ since $\mathcal H \subset \mathcal{U}$ by construction. Similarly, $\mathcal U\cap \mathcal{H}_0$ and $\mathcal U\cap \mathcal{H}_{tf}$ are dense in $\overline{\mathcal H}_{0}$ and $\overline{\mathcal H}_{tf}$, respectively. Thus $\mathcal{LH}$, $\mathcal{LH}_0$, and $\mathcal{LH}_{tf}$ are dense $G_\delta$-subsets of $\overline{\mathcal H}$, $\overline{\mathcal H}_{0}$, and $\overline{\mathcal H}_{tf}$, respectively. Clearly, $\overline{\mathcal H}$, $\overline{\mathcal H}_{0}$, and $\overline{\mathcal H}_{tf}$ are zero-dimensional Polish spaces being closed subsets of the zero-dimensional Polish space $\G$. Combining this with Lemma \ref{HLH} and Corollary \ref{CB}, we obtain that $\mathcal{LH}$, $\mathcal{LH}_0$, and $\mathcal{LH}_{tf}$ are homeomorphic to the Baire set and, in particular, to the subspace of irrational numbers in $\mathbb R$.
\end{proof}

\section{Remarks and open questions}
We conclude this paper with a brief discussion of open questions and directions for further research.
Our first question is motivated by the observation the all condensed groups discussed in this paper are not finitely presented.

\begin{prob}
Does there exist a finitely presented condensed group?
\end{prob}

Note that by Propositions \ref{cond} and \ref{enH}, this is equivalent to asking whether there exists a finitely presented extremely non-Hopfian group.

Here is yet another basic open question about condensed groups. For the study of the quasi-isometry relation in the relevant context, see \cite{MOW}.

\begin{prob}
Is the property of being condensed invariant under elementary equivalence of finitely generated groups? Is it geometric (i.e., invariant under quasi-isometry)?
\end{prob}

A group $G$ is called \emph{quasi-finitely-axiomatizable} (\emph{QFA}, for short) if there exists a first-order sentence $\sigma$ such that all finitely generated models of $\sigma$ are isomorphic to $G$ \cite{Nie}. Corollary \ref{LH-cor} shows that generic torsion-free lacunary hyperbolic groups are very far from being QFA.

\begin{prob}
Does there exist a non-cyclic torsion-free QFA lacunary hyperbolic group?
\end{prob}

Recall that the generic first-order theory $Th^{gen}(\S)$ of a subspace $\S\subseteq \G$ is the set of all first-order sentences in the language of groups whose set of models is comeager in $\S$. We denote by $Th^{gen}_\forall (\S)$ the subset of all universal sentences in $Th^{gen}(\S)$. Using methods of \cite{Osi09}, it is not difficult to show that $Th_{\forall}^{gen}(\overline{\mathcal H})$, $Th_{\forall}^{gen}(\overline{\mathcal H}_0)$, and $Th_{\forall}^{gen}(\overline{\mathcal H}_{tf})$ are undecidable. Hence, so are $Th^{gen}(\overline{\mathcal H})$, $Th^{gen}(\overline{\mathcal H}_0)$, and $Th^{gen}(\overline{\mathcal H}_{tf})$. The same result trivially holds for the closures of all other classes shown on diagram (\ref{diag}); indeed, the universal theory of each of them coincides with the universal theory of all groups or the universal theory of all torsion-free groups, both of which are known to be undecidable.

The general intuition suggests that every ``sufficiently complicated" subspace $\S\subseteq \G$ must have undecidable generic first-order theory. Is it possible to convert this into a precise statement? For example, we can ask the following.

\begin{prob}
Let $\S$ be a perfect subset of $\G$. Is it possible that $Th^{gen}(\S)$ is decidable?
\end{prob}

In Theorem \ref{hyp} (a), we showed that every comeager subset of $\mathcal{H}$ has infinitely many elementary equivalence classes but we do not know whether this number is countable or not.

\begin{prob}\label{card}
Does there exist a comeager subset of $\overline{\mathcal {H}}$ with countably many elementary equivalence classes?
\end{prob}

Note that if the answer to this question is negative, then the number of elementary equivalence classes in every comeager subset of $\overline{\mathcal {H}}$ is $2^{\aleph_0}$. Indeed, suppose that some comeager subset $\S\subseteq\overline{\mathcal {H}}$ has less than $2^{\aleph_0}$ equivalence classes. Let $\S_0$ be a dense $G_\delta $-subset of $\S$. Then $\S_0$ is Borel and, by the Silver dichotomy \cite{Sil}, the number of elementary equivalence classes in $\S_0$ must be countable. Clearly, $\S_0$ is also comeager in $\overline{\mathcal {H}}$.

Out next question is motivated by the highly non-constructive nature of the proof of Corollary \ref{LH-cor}.

\begin{prob}\label{lhex}
Find an explicit example of two non-isomorphic, elementarily equivalent, torsion-free, lacunary hyperbolic groups with property FA of Serre.
\end{prob}

Property FA is added here to rule out the ``trivial" examples coming from the hyperbolic world (e.g., $F_2$ and $F_3$). By Theorem \ref{Sela}, the required examples cannot be found among hyperbolic groups. On the other hand, we know that generic torsion-free lacunary hyperbolic groups satisfy all the requirements. Yet, we cannot provide any concrete examples!

Theorem \ref{LHgen} implies that the set $\mathcal {LH}$ is a Borel subset of $\G$. The set $\mathcal H$ is also Borel since it is countable. However, we do not know the answer to the following.

\begin{prob}
Are $\mathcal {RH}$ and $\mathcal {AH}$ Borel subsets of $\G$?
\end{prob}

Many of our results can be proved in the more general settings of $\Omega$-algebras, where $\Omega$ is any countable signature. We briefly outline possible generalizations and discuss related questions. For all unexplained notation and background from universal algebra, we refer to \cite{Cohn}.

An $\Omega$-algebra $A$ is \emph{generated} by a subset $Y=\{y_1, \ldots, y_n\}$ if every element $a\in A$ can be expressed as $a=t(y_1, \ldots, y_k)$ for some $\Omega$-term $t$. An \emph{$n$-generated marked $\Omega$-algebra} is a pair $(A,Y)$, where $A$ is an $\Omega$-algebra and $Y$ is an ordered generating set of $A$ of cardinality $n$. As in the case of groups, we consider such pairs up to the following equivalence relation: pairs $(A, (y_1, \ldots, y_n))$ and $(B, (z_1, \ldots , z_n))$ are identified if the map $y_i\mapsto z_i$, $i=1,\ldots, n$, extends to an isomorphism $A\to B$. We denote the set of all $n$-generated marked $\Omega $-algebras by $\mathcal A_n(\Omega)$.

A \emph{congruence} on an $\Omega$-algebra $A$ is a subset of $A\times A$ that is simultaneously an equivalence relation and an $\Omega $-subalgebra of $A\times A$. Let $\Con (A)$ denote the set of all congruences of $A$.  We also denote by $T_n=T(x_1, \ldots, x_n)$ the free $\Omega$-algebra generated by $\{ x_1, \ldots, x_n\}$. That is, $T_n$ is the set of all $\Omega$-terms in variables $x_1, \ldots, x_n$ equipped with the obvious operations.  It is well-known (and easy to prove) that elements of ${\rm Con}(A)$ are in one-to-one correspondence with homomorphisms of $\Omega$-algebras with domain $A$, see \cite[Section II.3]{Cohn}. This allows us to construct a bijection $\mathcal A_n(\Omega)\to\Con(T_n)$ in the same way as in Proposition \ref{GriCha}. The product topology on $2^{A\times A}$ induces the structure of a Hausdorff, zero-dimensional, separable, compact space on $\Con(T_n)$ and, via the above bijection, on $\mathcal A_n(\Omega)$.

Unlike in the case of groups, there is no ``canonical" embedding of $\mathcal A_n(\Omega)$ in $\mathcal A_{n+1}(\Omega)$. However, we can still form the \emph{space of finitely generated marked $\Omega$-algebras}  by taking the topological disjoint union
$$
\mathcal A (\Omega) = \bigsqcup_{n\in \NN} \mathcal A_n(\Omega).
$$
It is easy to see that $\mathcal A (\Omega)$ is a Polish space.

Let $\L$ be the first-order language with signature $\Omega$. The isomorphism class of a given $\Omega $-algebra, the set of models of a given $\L_{\omega_1,\omega}$-sentence, and satisfiability of the zero-one law for $\L_{\omega_1, \omega}$-sentences in a subspace $\S\subseteq \mathcal A(\Omega)$ are defined in the same way as in the case of groups. Once this terminology is established, one can show that the direct analogue of Theorem \ref{01} for $\Omega$-algebras holds.

\begin{thm}\label{01A}
For any isomorphism-invariant closed subspace $\S\subseteq \mathcal A (\Omega)$, the following conditions are equivalent.
\begin{enumerate}
\item[(a)] For any non-empty open sets $U$, $V$ in $\S$, there is a finitely generated $\Omega$-algebra $A$ such that $[A]\cap U\ne \emptyset$ and $[A]\cap V\ne \emptyset$.
\item[(b)] There exists a finitely generated $\Omega$-algebra $A$ such that ${[A]}$ is dense in $\S$.
\item[(c)] $\S $ satisfies the zero-one law for $\mathcal L_{\omega_1, \omega}$-sentences.
\end{enumerate}
\end{thm}

The proof is essentially the same as in the case of groups. The only difference is that, in general, there is no natural group of homeomorphisms of $\mathcal A (\Omega)$ whose orbits are precisely the isomorphism classes. However, it is not difficult to show that for any marked $\Omega$-algebras $(A,Y), (B,Z)\in \mathcal A (\Omega)$ such that $A\cong B$, there are open neighborhoods $N(A,Y)$ and $N(B,Z)$ of $(A,Y)$ and $(B,Z)$, respectively, and an isomorphism-preserving homeomorphism $N(A,Y)\to N(B,Z)$ taking $(A,Y)$ to $(B,Z)$. The existence of such local homeomorphisms is, in fact, sufficient to carry out the proof.

One can also generalize some other results of our paper in these settings. However, it is not quite clear whether this general theory is worth developing as it lacks non-trivial natural examples, like the spaces from Theorem \ref{hyp}.

\begin{prob}\label{pn}
Find interesting examples of countable signatures $\Omega$ and perfect subsets $\S\in \mathcal A(\Omega)$ containing dense isomorphism classes.
\end{prob}

A natural approach to Question \ref{pn} is based on condensed $\Omega$-algebras. As in the case of groups, we call an $\Omega$-algebra $U$ \emph{condensed} if its isomorphism class in $\mathcal A(\Omega)$ has no isolated points. It is then easy to show that for any finitely generated condensed group $G$ and any finitely generated ring $R$, the group ring $R[G]$ is finitely generated and condensed as a structure in the language of unital rings, whose signature is $\{ 0, 1, +, \cdot\}$. One can play around this idea a bit but all the examples we get this way are not very convincing as they essentially come from group theory.
Similarly to the case of groups, one can show that if a finitely generated $\Omega$-algebra $U$ is isomorphic to $U\times U$, then $U$ is condensed. We do not know the answer to the following.

\begin{prob}
Does there exist a non-zero, finitely generated, associative ring $R$ such that $R\cong R\times R$?
\end{prob}

\noindent \textbf{Denis Osin: } Department of Mathematics, Vanderbilt University, Nashville 37240, U.S.A.\\
E-mail: \emph{denis.v.osin@vanderbilt.edu}

\end{document}